\documentclass{amsart}
\usepackage{amssymb,amstext,amsmath,amscd,amsthm,amsfonts,enumerate,graphicx,latexsym,color,stmaryrd,geometry,multicol,ascmac,mathtools}
\usepackage[all]{xy}
\usepackage{comment}
\newtheorem{thm}{Theorem}[section]
\newtheorem*{thm*}{Theorem}

\newtheorem{cor}[thm]{Corollary}
\newtheorem{lem}[thm]{Lemma}
\newtheorem{prop}[thm]{Proposition}
\theoremstyle{definition}
\newtheorem{conv}[thm]{Convention}
\newtheorem{dfn}[thm]{Definition}
\newtheorem*{dfn*}{Definition}
\newtheorem{rem}[thm]{Remark}
\newtheorem{ques}[thm]{Question}

\newtheorem*{conj*}{Conjecture}
\newtheorem{ex}[thm]{Example}
\newtheorem{nota}[thm]{Notation}
\theoremstyle{remark}
\newtheorem*{ac}{Acknowledgments}
\newtheorem{claim}{Claim}
\newtheorem*{claim*}{Claim}
\renewcommand{\qedsymbol}{$\blacksquare$}
\numberwithin{equation}{thm}
\def\A{\mathcal{A}}
\def\a{\mathfrak{a}}

\def\ann{\operatorname{ann}}
\def\Ass{\operatorname{\mathrm{Ass}}}
\def\b{\mathfrak{b}}
\def\BB{\mathrm{B}}

\def\C{\mathcal{C}}

\def\cl{\operatorname{cl}}
\def\cof{\operatorname{\mathsf{Cof}}}

\def\D{\mathsf{D}}
\def\d{\mathrm{D}}
\def\Db{\mathsf{D^b}}

\def\Dp{\mathsf{D^+}}

\def\E{\mathbb{E}}
\def\e{\mathrm{E}}
\def\EE{\mathbf{E}}
\def\ee{{}^{\mathsf{s}}\mathbb{E}}
\def\Ext{\mathrm{Ext}}
\def\fid{\mathrm{fid}}
\def\ge{\geqslant}
\def\gldim{\operatorname{gldim}}
\def\grade{\operatorname{grade}}
\def\H{\mathrm{H}}
\def\hh{\operatorname{\mathcal{H}}}
\def\height{\operatorname{ht}}
\def\holim{\operatorname{\mathrm{holim}}}
\def\Hom{\mathrm{Hom}}

\def\Inj{\operatorname{\mathsf{Inj}}}

\def\le{\leqslant}
\def\LL{\mathbf{L}}
\def\ltensor{\otimes^{\bf{L}}}
\def\M{\mathcal{M}}

\def\Max{\operatorname{Max}}
\def\Mod{\operatorname{\mathsf{Mod}}}
\def\NN{\mathbb{N}}
\def\p{\mathfrak{p}}

\def\q{\mathfrak{q}}

\def\rad{\operatorname{rad}}

\def\Spec{\operatorname{Spec}}
\def\SS{\mathbf{S}}
\def\supp{\operatorname{supp}}
\def\Supp{\operatorname{Supp}}

\def\Tor{\mathrm{Tor}}
\def\V{\mathrm{V}}
\def\X{\mathcal{X}}
\def\xx{\boldsymbol{x}}
\def\Y{\mathcal{Y}}
\def\Z{\mathbb{Z}}
\def\ZZ{\mathrm{Z}}
\begin{document}
\title{Filtrations in module categories, derived categories and prime spectra}
\author{Hiroki Matsui}
\address[H. Matsui]{Graduate School of Mathematical Sciences\\ University of Tokyo, 3-8-1 Komaba, Meguro-ku, Tokyo 153-8914, Japan}
\email{mhiroki@ms.u-tokyo.ac.jp}
\author{Ryo Takahashi}
\address[R. Takahashi]{Graduate School of Mathematics, Nagoya University, Furocho, Chikusaku, Nagoya 464-8602, Japan}
\email{takahashi@math.nagoya-u.ac.jp}
\urladdr{http://www.math.nagoya-u.ac.jp/~takahashi/}
\thanks{Matsui was partly supported by JSPS Grant-in-Aid for JSPS Fellows 19J00158.
Takahashi was partly supported by JSPS Grants-in-Aid for Scientific Research 16K05098 and JSPS Fund for the Promotion of Joint International Research 16KK0099.}
\begin{abstract}
Let $R$ be a commutative noetherian ring.
The notion of $n$-wide subcategories of $\Mod R$ is introduced and studied in Matsui--Nam--Takahashi--Tri--Yen in relation to the cohomological dimension of a specialization-closed subset of $\Spec R$.
In this paper, we introduce the notions of $n$-coherent subsets of $\Spec R$ and $n$-uniform subcategories of $\D(\Mod R)$, and explore their interactions with $n$-wide subcategories of $\Mod R$.
We obtain a commutative diagram which yields filtrations of subcategories of $\Mod R$, $\D(\Mod R)$ and subsets of $\Spec R$ and complements classification theorems of subcategories due to Gabriel, Krause, Neeman, Takahashi and Angeleri H\"ugel--Marks--{\v S}{\v{t}}ov{\'{\i}}{\v{c}}ek--Takahashi--Vit\'{o}ria.
\end{abstract}
\maketitle
\tableofcontents
%%%%%%%%%%%%%%%%%%%%%%%%%%%%%%%%%%%%%%%%%%%%%%%%%%%%%%%%%%%%%%%%%
\section{Introduction}

A {\em localizing subcategory} of an abelian category is defined to be a full subcategory closed under coproducts, extensions, subobjects and quotient objects.
This notion was introduced by Gabriel \cite{Gab} to study the reconstruction problem of a noetherian scheme from the category of quasi-coherent sheaves.
Since then, localizing subcategories of an abelian category have been explored by a lot of authors to classify them and study abelian categories geometrically; see \cite{Gar,GP1,GP2,Her,K,hcls}.
A {\em wide subcategory} of an abelian category is by definition a full subcategory closed under extensions, kernels and cokernels.
In recent years, wide subcategories have actively been studied in representation theory of algebras; see \cite{AP,BM,Kra08,MS,wide,Y}.

A {\em localizing subcategory} of a triangulated category is defined as a full triangulated subcategory closed under coproducts.
As with localizing subcategories of an abelian category, one of the main topics in the study of localizing subcategories of a triangulated category is to classify them.
The first classification theorem has been obtained by Neeman \cite{Nee} for the unbounded derived category of a commutative noetherian ring.
Nowadays, localizing subcategories of a triangulated category are widely and deeply investigated in many areas of mathematics; see \cite{BCR, BIK, Dla1, Dla2, HS} for instance.

Let $R$ be a commutative noetherian ring.
Gabriel \cite{Gab} gives a complete classification of the localizing subcategories of the module category $\Mod R$ by specialization-closed subsets of $\Spec R$.
Krause \cite{Kra08} introduces the notion of a coherent subset of $\Spec R$ to classify the wide subcategories of $\Mod R$ closed under direct sums, which extends Gabriel's classification theorem.
Takahashi \cite{hcls} classifies the E-stable\footnote{This means $\ee$-closed in our sense.} subcategories of $\Mod R$ closed under direct sums and summands, extending Krause's classification theorem.

In \cite{cd} the notion of an $n$-wide subcategory of $\Mod R$ is introduced for each $n\in\NN$.
In this paper we extend this to $n=\infty$.
Also, we introduce the notions of an $n$-uniform subcategory of the derived category $\D(\Mod R)$ of $\Mod R$, and an $n$-coherent subset of $\Spec R$ for each $n\in\NN\cup\{\infty\}$.
The {\em (small) support} in the sense of \cite{Fox}, denoted by $\supp$, is our fundamental tool.
We prove the following theorem; the notation and terminology are explained in Convention \ref{2.1} and Definitions \ref{2.2}, \ref{2.4}, \ref{3.1}, \ref{3.4}, \ref{5.1}, \ref{5.3}, \ref{6.3}.

\begin{thm}[Theorems \ref{7}, \ref{8-2}, \ref{31} and Corollaries \ref{36'}, \ref{35}, \ref{25}, \ref{34}]\label{10a} 
Let $R$ be a commutative noetherian ring, and let $n\in\NN\cup\{\infty\}$.
Define the sets
$$
\begin{array}{ll}
\EE=\{\text{$(\ee,\oplus,\ominus)$-closed subcategories of $\Mod R$}\}, & \EE_n=\{\text{$n$-wide $(\E,\oplus)$-closed subcategories of $\Mod R$}\},\\
\SS=\{\text{subsets of $\Spec R$}\}, & \SS_n=\{\text{$n$-coherent subsets of $\Spec R$}\},\\
\LL=\{\text{localizing subcategories of $\D(\Mod R)$}\}, & \LL_n=\{\text{$n$-uniform localizing subcategories of $\D(\Mod R)$}\},\\
\EE'=\{\text{$(\ee,\oplus,\ominus,\Pi)$-closed subcategories of $\Mod R$}\}, & \EE'_n=\{\text{$n$-wide $(\E,\oplus,\Pi)$-closed subcategories of $\Mod R$}\},\\
\SS'=\{\text{generalization-closed subsets of $\Spec R$}\}, & \SS'_n=\{\text{$n$-coherent generalization-closed subsets of $\Spec R$}\},\\
\LL'=\{\text{bilocalizing subcategories of $\D(\Mod R)$}\}, & \LL'_n=\{\text{$n$-uniform bilocalizing subcategories of $\D(\Mod R)$}\}.
\end{array}
$$
Then there is a commutative diagram
$$
\xymatrix@C3.5pc{
\EE \ar@{=}[r] \ar[d]_{\supp}^\cong & \ \EE_\infty\ar[d]_{\supp}^\cong & \ \cdots\ar@{_{(}->}[l] & \ \EE_n\ar@{_{(}->}[l]\ar[d]_{\supp}^\cong & \ \cdots \ar@{_{(}->}[l] & \ \EE_2\ar[d]_{\supp}^\cong\ar@{_{(}->}[l] & \ \EE_1\ar[d]_{\supp}^\cong\ar@{_{(}->}[l] & \ \EE_0\ar[d]_{\supp}^\cong\ar@{_{(}->}[l]\\
\SS\ar@{=}[r] & \ \SS_\infty & \ \cdots\ar@{_{(}->}[l] & \ \SS_n\ar@{_{(}->}[l] & \ \cdots\ar@{_{(}->}[l] & \ \SS_2\ar@{_{(}->}[l] & \ \SS_1\ar@{_{(}->}[l] & \ \SS_0\ar@{_{(}->}[l]\\
\LL\ar@{=}[r]\ar[u]^\supp_\cong & \ \LL_\infty\ar[u]^\supp_\cong & \ \cdots\ar@{_{(}->}[l] & \ \LL_n\ar[u]^\supp_\cong\ar@{_{(}->}[l] & \ \cdots\ar@{_{(}->}[l] & \ \LL_2\ar[u]^\supp_\cong\ar@{_{(}->}[l] & \ \LL_1\ar[u]^\supp_\cong\ar@{_{(}->}[l] & \ \LL_0\ar[u]^\supp_\cong\ar@{_{(}->}[l]
}
$$
of sets, where the horizontal arrows are inclusion maps and the vertical arrows given by $\X\mapsto\supp\X=\bigcup_{X\in\X}\supp X$ are bijections.
This diagram restricts to the one where all the $\EE_*,\SS_*,\LL_*$ are replaced with $\EE_*',\SS_*',\LL_*'$ respectively.
One has
$$
\begin{array}{ll}
\EE_1=\{\text{$\oplus$-closed wide subcategories of $\Mod R$}\}, & \EE_0=\{\text{localizing subcategories of $\Mod R$}\},\\
\SS_1=\{\text{coherent subsets of $\Spec R$}\}, & \SS_0=\{\text{specialization-closed subsets of $\Spec R$}\},\\
\LL_1=\{\text{cohomology-closed localizing subcategories of $\D(\Mod R)$}\}, & \LL_0=\{\text{smashing subcategories of $\D(\Mod R)$}\},\\
\EE'_1=\{\text{bireflective Giraud subcategories of $\Mod R$}\}, & \EE'_0=\{\text{bilocalizing subcategories of $\Mod R$}\},\\
\SS'_1=\{\text{generalization-closed coherent subsets of $\Spec R$}\}, & \SS'_0=\{\text{clopen subsets of $\Spec R$}\},\\
\LL'_1=\{\text{cohomology-closed bilocalizing subcategories of $\D(\Mod R)$}\}, & \LL'_0=\{\text{bismashing subcategories of $\D(\Mod R)$}\}.
\end{array}
$$
\end{thm}

The bijections $\EE_0\to\SS_0$, $\EE\to\SS$, $\SS_0\gets\LL_0$, $\SS\gets\LL$, $\EE_1\to\SS_1\gets\LL_1$, and $\EE'_1\to\SS'_1\gets\LL'_1$ in Theorem \ref{10a} are respectively the same as the ones given by Gabriel \cite{Gab}, Takahashi \cite{hcls}, Neeman \cite{Nee}, Neeman \cite{Nee}, Krause \cite{Kra08}, and Angeleri H\"{u}gel, Marks, {\v S}{\v{t}}ov{\'{\i}}{\v{c}}ek, Takahashi and Vit\'{o}ria \cite{epi}.
By restricting the diagram in the theorem in other ways, we also give classifications of certain thick subcategories of $\Db(\Mod R)$, $\Dp(\Mod R)$ and $\D(\Mod R)_{\fid}$, the last one of which denotes the full subcategory of $\D(\Mod R)$ consisting of complexes of finite injective dimension.

The organization of this paper is as follows.
Section 2 is devoted to recalling basic facts on minimal injective resolutions and supports of modules/complexes.
In Section 3, we introduce the notions of $n$-wide subcategories of $\Mod R$, $n$-coherent subsets of $\Spec R$ for $n \in \NN \cup \{\infty\}$ and investigate fundamental properties of them.
In Section 4, we give a classification of $n$-wide subcategories of $\Mod R$ via $n$-coherent subsets of $\Spec R$, which contains results in \cite{Gab, Kra08, hcls}.
In Section 5, we introduce the notions of $n$-uniform and $n$-consistent subcategories of $\D(\Mod R)$ and classify some of them, which contains results in \cite{Kra08, Nee}. 
In Section 6, we consider two kinds of restrictions of classifications obtained in the previous two sections.
The first one is the restriction to classifications of $\Pi$-closed subcategories and hence the diagram in Theorem \ref{10a} is completed here.
The second one concerns classifications of $n$-uniform subcategories of $\D(\Mod R)_{\fid}$, $\Db(\Mod R)$ and $\Dp(\Mod R)$.
As an application we obtain a higher-dimensional analogue of Br\"{u}ning's classification theorem \cite{Br}.
In Section 7, we apply the classification of $n$-wide subcategories to a problem presented by Hartshorne \cite{H}.
We consider a weakened version of the notion of cofinite modules and study wideness of the subcategory of those modules.

%%%%%%%%%%%%%%%%%%%%%%%%%%%%%%%%%%%%%%%%%%%%%%%%%%%%%%%%%%%%%%%%%
\section{Preliminaries}

In this section, we recall several definitions which play important roles throughout this paper.

\begin{conv}\label{2.1}
Throughout this paper, we use the following convention.
We assume that all rings are commutative and noetherian and all subcategories are full.
We set $\NN:=\Z_{\ge0}=\{0,1,2,\dots\}$.
Let $R$ be a ring.
We denote by $\Mod R$ the category of (all) $R$-modules and by $\D(\Mod R)$ the (unbounded) derived category of the abelian category $\Mod R$.
We denote by $\e_R(M)$ the injective hull of an $R$-module $M$ and set $\kappa(\p)=R_\p/\p R_\p$ for each $\p\in\Spec R$.
We set $n+\infty=\infty$ and $n-\infty=-\infty$ for all $n\in\Z$.
We may simply say that a subcategory is $\oplus$-closed (resp. $\ominus$-closed, $\Pi$-closed) to mean that it is closed under (existing) direct sums (resp. direct summands, (existing) direct products)\footnote{In this paper, we say that a subcategory $\X$ of an additive category $\C$ is closed under direct sums (resp. direct products) provided that if $\{X_\lambda\}_{\lambda\in\Lambda}$ is a family of objects in $\X$ such that the direct sum $Y=\bigoplus_{\lambda\in\Lambda}X_\lambda$ (resp. the direct product $Y=\Pi_{\lambda\in\Lambda}X_\lambda$) exists in $\C$, then $Y$ belongs to $\X$.}.
We may omit subscripts and superscripts as long as there is no danger of confusion.
We may tacitly use statements given in Remarks.
\end{conv}

\begin{dfn}\label{2.2}
\begin{enumerate}[\rm(1)]		
\item
Let $\X$ be a subcategory of $\Mod R$.
We say that $\X$ is {\it Serre} if it is closed under extensions, submodules and quotient modules.
We say that $\X$ is {\it localizing} if it is Serre and closed under direct sums, and that $\X$ is {\it bilocalizing} if it is localizing and closed under direct products.
We say that $\X$ is {\it wide} if it is closed under extensions, kernels and cokernels.
\item
Let $\X$ be a subcategory of $\D(\Mod R)$.
We say that $\X$ is {\it thick} if it is closed under extensions, shifts and direct summands.
We say that $\X$ is {\it localizing} if it is thick and closed under direct sums, and that $\X$ is {\it bilocalizing} if it is localizing and closed under direct products.
\end{enumerate}
\end{dfn}

\begin{rem}\label{26}
We define a {\it strictly localizing subcategory} of an abelian category $\A$ (resp. a triangulated category $\mathcal{T}$) as a Serre (resp. thick) subcategory $\X$ such that the quotient functor $\A \to \A/\X$ (resp. $\mathcal{T} \to \mathcal{T}/\X$) has a right adjoint, see \cite{Kra00}. 
In the situation where $\A = \Mod R$ (resp. $\mathcal{T} = \D(\Mod R)$) with $R$ being a commutative noetherian ring, the localizing subcategories and the strictly localizing subcategories are the same.
Indeed, since $\A = \Mod R$ is a Grothendieck category, this follows by \cite[Corollaire 1, 375 page]{Gab}.
For the category $\mathcal{T} = \D(\Mod R)$, it is shown in \cite[Theorem 2.8]{Nee} that every localizing subcategory is generated by a set of objects and hence every localizing subcategory is strictly localizing by \cite[Lemma 2.1]{BIK0}.
\end{rem}

\begin{dfn}\label{2.4}
\begin{enumerate}[(1)]
\item
For $X \in \D(\Mod R)$ we define the {\em (small) support} of $X$ by
$$
\supp_RX= \{\p \in \Spec R \mid X \ltensor_R \kappa(\p) \not\cong 0\}.
$$
\item
Let $X$ be a complex of $R$-modules with $\H^iX=0$ for $i\ll0$.
Then one can take a {\em minimal injective resolution}
$$
\E_R(X)=(0\to\cdots\xrightarrow{\partial^{i-1}}\E_R^i(X)\xrightarrow{\partial^i}\E_R^{i+1}(X)\xrightarrow{\partial^{i+1}}\cdots)
$$
of $X$, that is, a bounded below complex of injective $R$-modules quasi-isomorphic to $X$ such that $\E_R^i(X)$ is the injective hull of the kernel of the map $\partial^i$ for all $i$.
\item
Let $M$ be an $R$-module.
For an integer $i>0$, we denote by $\mho^iM$ the $i$th {\em cosyzygy} of $M$, that is, the image of the $(i-1)$st differential map in the minimal injective resolution of $M$.
We set $\mho^0M=M$ and $\mho M=\mho^1M$.
\item
We say that a subcategory $\X$ of $\Mod R$ is {\em $\E$-closed} if $\E^i(M)\in\X$ for all $M\in\X$ and $i \ge 0$.
This is equivalent to saying that each object $M\in\X$ admits an injective resolution $I$ with $I^i\in\X$ for all $i\ge 0$.
If in addition $\X$ satisfies $M \in \X$ whenever $\E^i(M) \in \X$ for all $i \ge 0$, we say that $\X$ is {\em $\ee$-closed}. 
\if0
\old{We say that a subcategory $\X$ of $\Mod R$ is {\em $\ee$-closed} if for an $R$-module $M$ one has $M\in\X$ if and only if $\E^i(M) \in \X$ for all $i \ge 0$.
We say that $\X$ is {\em $\E$-closed} if $\E^i(M)\in\X$ for all $M\in\X$ and $i \ge 0$.
This is equivalent to saying that each object $M\in\X$ admits an injective resolution $I$ with $I^i\in\X$ for all $i\ge 0$.}\fi
In a similar fashion, {\em $\ee$-closed} and {\em $\E$-closed} subcategories of $\Dp(\Mod R)$ are defined by using minimal injective resolutions of complexes.
\end{enumerate}
\end{dfn}

\begin{nota}
Let $\X$ be a subcategory of $\Mod R$ and $\Phi$ a subset of $\Spec R$.
For $\Delta\in\{\supp,\Ass\}$ we denote by $\Delta^{-1}_\X(\Phi)$ the subcategory of $\X$ consisting of objects $X$ with $\Delta(X)\subseteq\Phi$.
\end{nota}

\begin{rem}\cite[Proposition 2.3 and Corollary 2.5]{cd}\cite[Theorem 18.7]{M}\label{rem}
\begin{enumerate}[(1)]
\item
Let $X\in\D(\Mod R)$.
Let $I$ be a complex of injective $R$-modules quasi-isomorphic to $X$.
Then $\supp X\subseteq\bigcup_{i\in\Z}\Ass I^i$.
The equality holds if $\H^{\ll0}(X)=0$ and $I=\E(X)$.
In particular, for each $R$-module $M$ there is an inclusion $\Ass M\subseteq\supp M$, whose equality holds if $M$ is injective.
\item
Let $\Phi$ be a subset of $\Spec R$.
Then $\supp^{-1}_{\Mod R}(\Phi)$ is closed under direct sums, direct summands and extensions.
One also has $\supp^{-1}_{\Mod R}(\Phi) \subseteq\Ass^{-1}_{\Mod R}(\Phi)$.
\item
Let $M\in\Mod R$, $\p\in\Spec R$ and $i\in\Z$.
The number $\mu_i(\p, M) = \dim_{\kappa(\p)}\Ext_{R_\p}^i(\kappa(\p), M_\p)$ is equal to the cardinality of the number of direct summands $\e_R(R/\p)$ of $\E^i(M)$.
Thus, there is a direct sum decomposition into indecomposable injective modules $\E^i(M) \cong \bigoplus_{\p \in \Spec R} \e_R(R/\p)^{\oplus \mu_i(\p, M)}$.
\end{enumerate}
\end{rem}

%%%%%%%%%%%%%%%%%%%%%%%%%%%%%%%%%%%%%%%%%%%%%%%%%%%%%%%%%%%%%%%
\section{Basic properties of $n$-wide subcategories and $n$-coherent subsets}

In this section, we introduce the notion of $n$-coherent subsets of $\Spec R$ with $n \in \NN\cup \{\infty\}$, which form a filtration of classes of subsets of $\Spec R$.
The class of specialization-closed subsets and that of coherent subsets introduced in \cite{Kra08} appear among this filtration.
As a categorical counterpart of this filtration, we introduce the notion of $n$-wide subcategories of $\Mod R$ with $n\in\NN\cup \{\infty\}$, which has been introduced in \cite{cd} for $n \in \NN$.
Likewise, two important classes of subcategories, that is, those of Serre subcategories and wide subcategories, appear in this filtration.

First of all, we recall here the definition of an $n$-wide subcategory and extend it to $n=\infty$.

\begin{dfn}\label{3.1}
Let $n$ be a nonnegative integer (resp. $n=\infty$).
A subcategory $\X$ of $\Mod R$ is said to be closed under {\it $n$-kernels}\footnote{Our definitions of $n$-kernels and $n$-cokernels are different from those in \cite[Definition 2.2]{J}.} if for every exact sequence $0 \to M \to X^0 \to X^1 \to X^2 \to \cdots$ in $\Mod R$ with $X^i \in \X$ for all $0 \le i\le n$ (resp. all $i\ge0$) the module $M$ is in $\X$.
Dually, $\X$ is said to be closed under {\it $n$-cokernels} if for every exact sequence $\cdots \to X_2 \to X_1 \to X_0 \to M \to 0$ in $\Mod R$ with $X_i \in \X$ for all $0 \le i\le n$ (resp. all $i\ge0$) the module $M$ is in $\X$.
We say that a subcategory $\X$ of $\Mod R$ is {\it $n$-wide} if it is closed under extensions, $n$-kernels and $n$-cokernels.
\end{dfn}

\begin{rem}\label{rem2}
Let $\X$ be a subcategory of $\Mod R$.
\begin{enumerate}[(1)]
\item
Let $n\in\NN\cup\{\infty\}$.
If $\X$ is closed under $n$-kernels (resp. $n$-cokernels), then $\X$ is closed under $(n+1)$-kernels (resp. $(n+1)$-cokernels) and $\infty$-kernels (resp. $\infty$-cokernels).
In particular, any $n$-wide subcategory of $\Mod R$ is both $(n+1)$-wide and $\infty$-wide.
In general, there is an $(n+1)$-wide subcategory which is not $n$-wide, see the example below.
\item
$\X$ is closed under $0$-kernels (resp. $0$-cokernels) if and only if $\X$ is closed under submodules (resp. quotient modules).
In particular, $\X$ is $0$-wide and $\oplus$-closed if and only if it is localizing.
\item
$\X$ is closed under $1$-kernels (resp. $1$-cokernels) if and only if $\X$ is closed under kernels (resp. cokernels).
In particular, $\X$ is $1$-wide if and only if it is wide.
\item
Let $n\in\NN\cup\{\infty\}$.
If $\X$ is closed under $n$-kernels (resp. $n$-cokernels) e.g., if $\X$ is $n$-wide, then $\X$ is closed under direct summands and kernels of epimorphisms (resp.  cokernels of monomorphisms).
Indeed, for any $R$-modules $M,N$ the sequence $\cdots\xrightarrow{\left(\begin{smallmatrix}0&0\\0&1\end{smallmatrix}\right)}M\oplus N\xrightarrow{\left(\begin{smallmatrix}1&0\\0&0\end{smallmatrix}\right)}M\oplus N\xrightarrow{\left(\begin{smallmatrix}0&0\\0&1\end{smallmatrix}\right)}\cdots$ is exact, whose cocycles are $M,N$.
This shows that $\X$ is closed under direct summands.
For a short exact sequence $0 \to L \to M \to N \to 0$, one has the long exact sequence $0 \to L \to M \to N \to 0 \to \cdots \to 0 \to 0$.
Hence $\X$ is closed under kernels of epimorphisms provided $\X$ is closed under $n$-kernels and similarly for $\X$ being closed under $n$-cokernels.
\item
If $\X$ is closed under $\infty$-kernels and $\E$-closed, then $\X$ is $\ee$-closed.
In particular, every $n$-wide $\E$-closed subcategory of $\Mod R$ with $n\in\NN\cup\{\infty\}$ is $\ee$-closed.
\item
In \cite{NSZ}, they use the notation $\mathrm{Pres}^n(\X)$ for the subcategory of $R$-modules $M$ admitting a {\it presentation of length $n$} by objects of $\X$ i.e., an exact sequence $X_n \to X_{n-1} \to  \ldots \to X_0 \to M \to 0$ with $X_i \in \X$ for all $i = 0, 1, \ldots, n$. 
Using this notation, we have that $\X$ is closed under $n$-cokernels if and only if $\mathrm{Pres}^n(\X) \subseteq \X$.
Dually, using the subcategory $\mathrm{Copres}^n(\X)$ consisting of $R$-modules admitting {\it copresentation of length $n$} by objects of $\X$, we have that $\X$ is closed under $n$-kernels if and only if $\mathrm{Copres}^n(\X) \subseteq \X$.
\end{enumerate}
\end{rem}

\begin{ex}
\begin{enumerate}[\rm(1)]
\item
Fix an $R$-module $C$.
Then we define the $C$-{\it grade} of an $R$-module $M$ by
$
\mbox{$C$-$\grade_R(M)$} := \inf \{i \ge 0 \mid \Ext_R^i(C, M) \neq 0 \}.
$	
One can easily check that the subcategory of $\Mod R$ consisting of $R$-modules with $\mbox{$C$-$\grade_R(M)$} \le n$ is closed under $n$-kernels (cf. \cite[Proposition 1.2.9]{BH}).
Dually, the $C$-{\it cograde} of $M$ is defined by
$
\mbox{$C$-$\mathrm{cograde}_R(M)$} := \inf \{i \ge 0 \mid \Tor^R_i(C, M) \neq 0 \},
$	
see \cite[Corollary 3.11]{O}.
Then the subcategory of $\Mod R$ consisting of $R$-modules with $\mbox{$C$-$\mathrm{cograde}_R(M)$} \le n$ is closed under $n$-cokernels.

\item (\cite[Theorem 4.3]{cd})
Let $\a$ be an ideal of $R$ with cohomological dimension at most $n$, i.e., $\H^{>n}_\a (M) = 0$ for all $M$.
Then $\supp^{-1}_{\Mod R}(\mathrm{D}(\a))$ is $n$-wide.
For example, if $\a$ is generated by $n$-elements, then $\supp^{-1}_{\Mod R}(\mathrm{D}(\a))$ is $n$-wide.
Moreover, this subcategory is not $(n-1)$-wide provided $\a$ is generated by a regular sequence of length $n$. 
\end{enumerate}
\end{ex}

For a subset $\Phi$ of $\Spec R$ and a prime ideal $\p$ of $R$, we set
$$
\Phi_\p := \{P\in\Spec R_\p\mid P\cap R\in\Phi\}.
$$
Here, $P \cap R$ means the pullback of $P$ along the canonical ring homomorphism $R \to R_\p$.
Then $\Phi_\p$ is homeomorphic to the generalization-closure $\{\q \in \Phi \mid \q \subseteq \p\}$ of $\{\p\}$ in $\Phi$ by the assignment $P \mapsto P \cap R$.
The following fact is frequently used in this paper.

\begin{prop}\label{222}
Let $n$ be either a nonnegative integer or $\infty$.
Let $\Phi$ be a subset of $\Spec R$.
Then $\supp^{-1}_{\Mod R}(\Phi)$ is $n$-wide if and only if $\supp_{\Mod R_\p}^{-1}(\Phi_\p)$ is $n$-wide for each $\p \in \Spec R$.
\end{prop}

\begin{proof}
Note that for a multiplicatively closed subset $S$ of $R$ and an $R$-module $M$ one has $\supp_{R_S} M_S = \{P \in \Spec R_S \mid P \cap R \in \supp_R M\}$.
Hence $M \in \supp_{\Mod R}^{-1}(\Phi)$ if and only if $M_\p \in\supp_{\Mod R_\p}^{-1}(\Phi_\p)$ for all $\p \in \Spec R$.
The ``if'' part of the proposition is deduced from this.
We show the ``only if'' part from now on.
Let $n\in\NN$ (resp. $n=\infty$), fix a prime ideal $\p$ of $R$, and take an exact sequence $0 \to M \to X^0 \to X^1\to\cdots$ in $\Mod R_\p$ with $X^i \in \supp^{-1}_{\Mod R_\p}(\Phi_\p)$ for all $0\le i\le n$ (resp. for all $i\ge0$).
Then we see that $X^i\in\supp^{-1}_{\Mod R}(\Phi)$ for all $0\le i\le n$ (resp. for all $i\ge0$).
The $n$-wideness of $\supp^{-1}_{\Mod R}(\Phi)$ yields $M\in\supp^{-1}_{\Mod R}(\Phi)$, that is, $\supp_RM\subseteq\Phi$.
We then observe that $\supp_{R_\p}M\subseteq\Phi_\p$, namely, $M\in\supp^{-1}_{\Mod R_\p}(\Phi_\p)$.
This shows that $\supp^{-1}_{\Mod R_\p}(\Phi_\p)$ is closed under $n$-kernels.
Closure under $n$-cokernels is shown similarly.
\end{proof}

Next, we introduce the notion of $n$-coherent subsets of $\Spec R$.
Denote by $\Inj R$ the subcategory of $\Mod R$ consisting of injective $R$-modules.

\begin{dfn}\label{3.4}
Let $n$ be a nonnegative integer (resp. $n=\infty$).
A subset $\Phi$	of $\Spec R$ is called {\it $n$-coherent} if for an exact sequence $\cdots \to I_2 \to I_1 \to I_0 \to C \to 0$ in $\Mod R$ with $I_i\in\Inj R$ and $\Ass I_i \subseteq \Phi$ for all $0\le i\le n$ (resp. all $i\ge0$) the module $C$ is embedded in some $J \in\Inj R$ with $\Ass J \subseteq \Phi$.
\end{dfn}

\begin{rem}\label{4}
\begin{enumerate}[(1)]
\item
Let $n\in\NN$ (resp. $n=\infty$).
A subset $\Phi$ of $\Spec R$ is $n$-coherent if and only if for an exact sequence $\cdots \to I_1 \to I_0 \to C \to 0$ in $\Mod R$ with $I_i\in\Ass_{\Inj R}^{-1}(\Phi)$ for all $0\le i\le n$ (resp. all $i\ge0$) one has $\Ass_RC\subseteq\Phi$.
\item
Let $n\in\NN\cup\{\infty\}$.
Then $n$-coherence implies $(n+1)$-coherence and $\infty$-coherence.
In general, there is an $(n+1)$-coherent subset which is not $n$-coherent. 
\item
The $1$-coherent subsets are the same as the coherent subsets in the sense of \cite{Kra08}.
\item
Let $n\in\NN\cup\{\infty\}$.
If $\supp_{\Mod R}^{-1}(\Phi)$ is closed under $n$-cokernels, then $\Phi$ is $n$-coherent.
\end{enumerate}
\end{rem}

\begin{ex}(\cite[Example 4.11]{cd})
Let $\underline{x} := x_1, \ldots, x_n$ be a sequence of elements of $R$.
Then $\mathrm{D}(\underline{x}) := \Spec R \setminus \mathrm{V}(\underline{x})$ is $n$-coherent which is not $(n-1)$-coherent if $\underline{x}$ is an $R$-regular sequence.	
\end{ex}

The latter assertion of the following proposition includes \cite[Proposition 4.1(2)]{Kra08}.

\begin{prop}\label{3}
The $0$-coherent subsets are the same as the specialization-closed subsets.
In particular, specialization-closed subsets are coherent.
\end{prop}

\begin{proof}
The second assertion follows from the first one and Remark \ref{4}(2)(3).
To show the first assertion, let $\Phi$ be a subset of $\Spec R$.
Suppose that $\Phi$ is specialization-closed.
Then by \cite[Theorem 1.1]{cd} the subcategory $\Ass_{\Mod R}^{-1}(\Phi)$ is closed under quotient modules.
Let $I\twoheadrightarrow C$ be a surjective homomorphism with $I\in\Ass_{\Inj R}^{-1}(\Phi)$.
Then $I$ is in $\Ass_{\Mod R}^{-1}(\Phi)$, and so is $C$.
Hence $\e(C)\in\Ass_{\Inj R}^{-1}(\Phi)$, which shows that $\Phi$ is $0$-coherent.
Conversely, assume that $\Phi$ is $0$-coherent.
Let $M\in\Ass_{\Mod R}^{-1}(\Phi)$.
Then $\E^0(M)=\e(M)\in\Ass_{\Inj R}^{-1}(\Phi)$ and the surjection $\E^0(M)\twoheadrightarrow \mho M$ shows that $\mho M$ is embedded in a module in $\Ass_{\Inj R}^{-1}(\Phi)$, which implies $\E^1(M)=\e_R(\mho M)\in\Ass_{\Inj R}^{-1}(\Phi)$.
Iterating this procedure, we see that $\E^i(M)\in\Ass_{\Inj R}^{-1}(\Phi)$ for all $i\ge0$, and $M\in\supp^{-1}_{\Mod R}(\Phi)$.
Thus $\Ass_{\Mod R}^{-1}(\Phi)=\supp^{-1}_{\Mod R}(\Phi)$, and $\Phi$ is specialization-closed by \cite[Theorem 1.1]{cd}.
\end{proof}

The following extends \cite[Proposition 4.1(1)(4)]{Kra08} on coherent subsets to $n$-coherent ones.

\begin{prop}\label{2}
Let $n$ be either a nonnegative integer or $\infty$.
\begin{enumerate}[\rm(1)]
\item
Let $\{\Phi_\lambda\}_{\lambda\in\Lambda}$ be a family of $n$-coherent subsets of $\Spec R$.
Then $\bigcap_{\lambda \in \Lambda} \Phi_\lambda$ is $n$-coherent.
\item
A subset $\Phi$ of $\Spec R$ is $n$-coherent if and only if $\Phi_\p$ is $n$-coherent for each $\p \in \Spec R$.
\end{enumerate}
\end{prop}

\begin{proof}
Let $n$ be a nonnegative integer (resp. $n=\infty$).

(1) Let $\cdots \to I_2 \to I_1 \to I_0 \to C \to 0$ be an exact sequence with $I_i\in\Ass_{\Inj R}^{-1}(\bigcap_{\lambda\in\Lambda}\Phi_\lambda)$ for all $0 \le i \le n$ (resp. all $i\ge0$).
Then by assumption, $C$ is embedded in some $I_\lambda\in\Ass_{\Inj R}^{-1}(\Phi_\lambda)$ for each $\lambda\in\Lambda$.
Hence $\e_R(C)\in\Ass_{\Inj R}^{-1}(\Phi_\lambda)$ for all $\lambda\in\Lambda$, that is, $\e_R(C)\in\Ass_{\Inj R}^{-1}(\bigcap_{\lambda\in\Lambda}\Phi_\lambda)$.

(2) First of all, recall the fundamental fact that for a multiplicatively closed subset $S$ of $R$ one has $(\Inj R)_S\subseteq\Inj(R_S)\subseteq\Inj R$ and $\Ass_{R_S}(M_S)=\{P\in\Spec R_S\mid P\cap R\in\Ass_RM\}$ for $M\in\Mod R$ (see \cite[Theorem 6.2]{M}).

The ``only if'' part:
Fix $\p\in\Spec R$.
Let $\cdots \to I_2 \to I_1 \to I_0 \to C \to 0$ be an exact sequence of $R_\p$-modules with $I_i\in\Ass_{\Inj R_\p}^{-1}(\Phi_\p)$ for all $0 \le i\le n$ (resp. all $i\ge0$).
Then $I_i\in\Ass_{\Inj R}^{-1}(\Phi)$ for all $0 \le i\le n$ (resp. all $i\ge0$), and the $n$-coherence of $\Phi$ implies that $C$ is embedded in some $J\in\Ass_{\Inj R}^{-1}(\Phi)$.
Hence $C=C_\p$ is embedded in $J_\p\in\Ass_{\Inj R_\p}^{-1}(\Phi_\p)$.

The ``if'' part:
Let $\cdots \to I_2 \to I_1 \to I_0 \to C \to 0$ be an exact sequence of $R$-modules with $I_i\in\Ass_{\Inj R}^{-1}(\Phi)$ for all $0 \le i \le n$ (resp. all $i\ge0$).
For each $\p\in\Spec R$ the exact sequence $\cdots \to (I_2)_\p \to (I_1)_\p \to (I_0)_\p \to C_\p \to 0$ of $R_\p$-modules is induced, and $(I_i)_\p\in\Ass_{\Inj R_\p}^{-1}(\Phi_\p)$ for all $0 \le i \le n$ (resp. all $i\ge0$).
Since $\Phi_\p$ is $n$-coherent, $C_\p$ is embedded in some $J(\p)\in\Ass_{\Inj R_\p}^{-1}(\Phi_\p)$.
Hence $\e_{R_\p}(C_\p)\in\Ass_{\Inj R_\p}^{-1}(\Phi_\p)$ for all $\p\in\Spec R$.
It is easy to see that $\e_R(C)\in\Ass_{\Inj R}^{-1}(\Phi)$.
\end{proof}

Let $\Phi$ be a specialization-closed subset of $\Spec R$, $i$ an integer and $M$ an $R$-module.
We denote by $\H_\Phi^i(M)$ the $i$th {\em local cohomology} module of $M$ with respect to $\Phi$.
The details of local cohomology with respect to a specialization-closed subset are found in \cite[\S3]{cd} for instance.
Here we relate $n$-coherence to local cohomology.

\begin{prop}\label{rigid}
Let $\Phi$ be a specialization-closed subset of $\Spec R$, and let $n\ge0$ be an integer.
Then $\Phi^\complement$ is $n$-coherent if and only if $\H_\Phi^{>n}(M) = 0$ for all $R$-modules $M$ with $\H_\Phi^{\le n}(M) = 0$.
\end{prop}

\begin{proof}
To show the ``if'' part, take an exact sequence $0\to K\xrightarrow{f^0}I^0\xrightarrow{f^1}\cdots\xrightarrow{f^n}I^n\xrightarrow{f^{n+1}}C\to0$ with $I^i \in \Ass_{\Inj R}^{-1}(\Phi^\complement)$ for all $0\le i\le n$.
It follows from \cite[Proposition 4.5]{cd} that $\H_\Phi^i(K) = 0$ for $0 \le i \le n$.
By assumption $\H_\Phi^{n+1}(K) = 0$.
Let $K^i$ be the image of $f^i$ for each $0\le i\le n+1$.
There is an exact sequence $0\to K^i\to I^i\to K^{i+1}\to0$ for each $0\le i\le n$.
We get $\H_\Phi^1(K^n)\cong\H_\Phi^2(K^{n-1})\cong\cdots\cong\H_\Phi^{n+1}(K^0)=0$.
Combining this with the equality $\Gamma_\Phi(I^n)=0$ from \cite[Proposition 3.2(1)]{cd}, we see that $\Gamma_\Phi(K^{n+1})=0$.
The module $C=K^{n+1}$ is embedded in $\e_R(K^{n+1})$, which belongs to $\Ass_{\Inj R}^{-1}(\Phi^\complement)$.

We show the ``only if'' part.
By Remark \ref{4}(2), it suffices to verify $\H_\Phi^{n+1}(M)=0$ for an $R$-module $M$ with $\H_\Phi^{\le n}(M) = 0$.
By \cite[Proposition 4.5]{cd} we get an exact sequence $0\to M\to I^0\to\cdots\to I^n$ with $I^i\in\Ass_{\Inj R}^{-1}(\Phi^\complement)$ for all $i\le n$.
As $\Phi$ is $n$-coherent, this is extended to an exact sequence $0\to M\to I^0\to\cdots\to I^n\to I^{n+1}$ with $I^{n+1}\in\Ass_{\Inj R}^{-1}(\Phi^\complement)$.
Using \cite[Proposition 4.5]{cd} again, we obtain $\H_\Phi^{n+1}(M)=0$.
\end{proof}

\begin{rem}
For the case $n=1$, this proposition recovers \cite[Theorem 4.9 (2) $\Leftrightarrow$ (4)]{epi}.
Indeed, for an $R$-module $M$, there is an $R$-module $\mathrm{D}_\Phi(M)$ such that
$\H_\Phi^i(\mathrm{D}_\Phi(M)) = 0$ for $i=0,1$ and $\H_\Phi^i(M) \cong\H_\Phi^i(\mathrm{D}_\Phi(M))$ for $i >1$, see \cite[Proposition 3.9]{cd} for details.
Therefore, by the above proposition, $\Phi^\complement$ is coherent if and only if $\H_\Phi^{>1}(M) = 0$ for all $R$-modules $M$.
\end{rem}

%%%%%%%%%%%%%%%%%%%%%%%%%%%%%%%%%%%%%%%%%%%%%%%%%%
\section{Classification of $n$-wide $(\E,\oplus)$-closed subcategories of $\Mod R$}

In the previous section, we recalled/introduced the notions of an $n$-wide subcategory of $\Mod R$ and an $n$-coherent subset of $\Spec R$ and studied their basic properties.
The aim of this section is to explore the relationship between them.

We introduce a new series of subcategories of modules, and investigate some properties of them.

\begin{nota}
Let $\X$ be a subcategory of $\Mod R$.
For $n \in\NN$ we denote by $\C_\X^n$ the subcategory of $\Mod R$ consisting of modules $M$ admitting an exact sequence $0 \to M \to I^0 \to I^1 \to \cdots\to I^n$ with $I^i\in \Inj R \cap \X$ for all $i$.
Set $\C_\X^{-1} = \Mod R$.
When we consider $\X := \Ass_{\Mod R}^{-1}(\Phi)$ for a subset $\Phi$ of $\Spec R$, we set $\C_\Phi^n := \C_\X^n$ for each $n$.
\end{nota}

Here, we list basic properties of the above series of subcategories, which easily follow from the definition.

\begin{rem}\label{8}
Let $n\ge0$ be an integer, $\X$ a subcategory of $\Mod R$ and $\Phi$ a subset of $\Spec R$.
\begin{enumerate}[(1)]
\item
The subset $\Phi$ is $n$-coherent if and only if $\C_{\Phi}^n = \C_{\Phi}^{i}$ for all $i \ge n$.
\item
There is a filtration $\Mod R=\C_\X^{-1}\supseteq\C_\X^0 \supseteq \C_\X^1 \supseteq \cdots \supseteq \bigcap_{i\ge 0} \C_\X^i$ of subcategories, and the equality $\bigcap_{i\ge0}\C_\Phi^i = \supp_{\Mod R}^{-1}(\Phi)$ holds.
In particular, $\Phi$ is $n$-coherent if and only if	 $\supp_{\Mod R}^{-1}(\Phi)= \C_\Phi^n$.
\item
If $\X$ is $\E$-closed and contains $\bigcap_{i\ge0}\C_\X^i$, then $\X$ is $\ee$-closed.
\item
Suppose that $\X$ is closed under direct summands.
Then the following hold.
\begin{enumerate}[(a)]
\item
An $R$-module $M$ is in $\C_\X^n$ if and only if $\E^i(M) \in \X$ for all $0 \le i \le n$.
\item
The subcategory $\X$ is $\E$-closed if and only if $\X\subseteq\bigcap_{i\ge0}\C_\X^i$.
Hence, $\X$ is $\ee$-closed if and only if $\X=\bigcap_{i\ge0}\C_\X^i$.
\end{enumerate}
\end{enumerate}
\end{rem}

\begin{lem}\label{ex}
Let $\X$ be a subcategory of $\Mod R$ closed under finite direct sums.
Let $0 \to L \xrightarrow{f} M \xrightarrow{g} N \to 0$ be an exact sequence of $R$-modules.
Let $n\ge0$ be an integer.
Then the first two of the following implications hold true, and so is the third if $\X$ is closed under direct summands.
\begin{enumerate}[\rm(i)]
\item
$L \in \C_\X^n,\,N \in \C_\X^{n}\Rightarrow M \in \C_\X^n$,
\item
$M \in \C_\X^n,\,N \in \C_\X^{n-1}\Rightarrow L \in \C_\X^n$,
\item
$L \in \C_\X^n,\,M \in \C_\X^{n-1}\Rightarrow N \in \C_\X^{n-1}$.
\end{enumerate}
\end{lem}

\begin{proof}
The first implication directly follows from the horseshoe lemma.
To show the second, take a lift $I\to J$ of $g$ to injective resolutions of $M$ and $N$ with $I^i,J^i\in\X$ for all $0\le i\le n$ and $0\le j\le n-1$.
Taking the mapping cone, we get an injective resolution $(0\to I^0\to J^0\oplus I^1\to\cdots\to J^{n-1}\oplus I^n\to\cdots)$ of the module $L$, whose first $n+1$ terms are in $\X$.
This shows $L \in \C_\X^n$.
The third implication is similarly shown: taking a lift $I\to J$ of $f$ to injective resolutions of $L$ and $M$ and the mapping cone, we get an injective resolution $(0\to(J^0\oplus I^1)/I^0\to J^1\oplus I^2\to\cdots\to J^{n-1}\oplus I^n\to\cdots)$ of $N$ whose first $n$ terms are in $\X$, if $\X$ is closed under direct summands.
Here, the injectivity of $(J^0\oplus I^1)/I^0$ follows since $I^0$ is injective and $I^0 \to J^0\oplus I^1$ is a monomorphism.
\end{proof}

\begin{rem}
This lemma generalizes the depth lemma.
Indeed, if we consider $\Phi:= \d(\a)$ with $\a$ an ideal of $R$, then $\C_\Phi^n$ is nothing but the class of $R$-modules $M$ with $\grade(\a,M) > n$ in the sense of \cite[Definition 9.1.1]{BH} by \cite[Theorem 2.1]{FI} and \cite[Proposition 4.5]{cd}.
\end{rem}

The following result gives a way to construct an $n$-wide subcategory from a given $n$-coherent subset.

\begin{prop}\label{9}
Let $n\ge0$ be an integer, and let $\Phi$ be a subset of $\Spec R$.
\begin{enumerate}[\rm(1)]
\item
The subcategory $\C_\Phi^n$ of $\Mod R$ is closed under direct sums, extensions and $n$-kernels.
\item
If $\Phi$ is $n$-coherent, then $\supp_{\Mod R}^{-1}(\Phi) = \C_\Phi^n$ is an $n$-wide $(\ee,\oplus)$-closed subcategory of $\Mod R$.
\end{enumerate} 
\end{prop}

\begin{proof}
(1) Let $0 \to M \xrightarrow{f_0} X_0 \xrightarrow{f_1}\cdots \xrightarrow{f_n} X_n$ be an exact sequence with $X_i \in \C_\Phi^n$ for all $i$.
Let $U_i$ be the image of $f_i$.
The inclusions $U_n \hookrightarrow X_n \hookrightarrow \e_R(X_n)$ show $U_n \in \C_\Phi^0$.
Using the second implication in Lemma \ref{ex} inductively, we are done.

(2) The equality follows from Remark \ref{8}(2).
By (1) it suffices to show that $\supp_{\Mod R}^{-1}(\Phi)$ is closed under $n$-cokernels.	
Consider an exact sequence $0 \to M \to X_0 \to X_1 \to \cdots \to X_n \to N \to 0$ with $X_i \in \supp_{\Mod R}^{-1}(\Phi)$ for all $i$.
Applying the third implication in Lemma \ref{ex} repeatedly and using Remark \ref{8}(2), we observe that $N$ belongs to $\supp_{\Mod R}^{-1}(\Phi)$.
\end{proof}

For an $R$-module $M$ we denote by $\dim_RM$ the {\em (Krull) dimension} of $M$, i.e., $\dim_RM=\sup\{\dim R/\p\mid\p\in\Supp_RM\}$.
The next proposition gives a sufficient condition for a subset of a given $n$-coherent subset to be again $n$-coherent.

\begin{prop}\label{923-1}
Let $n\ge 0$ be an integer.
Let $\Phi,\Psi$ be subsets of $\Spec R$, and assume that $\Psi$ is $n$-coherent.
Suppose either that {\rm(i)} $\height\p\le n$ for all $\p\in\Phi$ or that {\rm(ii)} $\dim R/\p\le n$ for all $\p\in\Psi$.
Then $\Psi\setminus\Phi$ is $n$-coherent.
\end{prop}

\begin{proof}
Let $0\to M\to I^0\to\cdots\to I^n\to N\to0$ be an exact sequence with $I^i\in\Ass_{\Inj R}^{-1}(\Psi\setminus\Phi)$ for all $0\le i\le n$.
We want to prove that $\Ass N\subseteq\Psi\setminus\Phi$.
Since $\Psi$ is $n$-coherent, we have $\Ass N\subseteq\Psi$, and it is enough to show that $\Ass N\subseteq\Phi^\complement$.
Take any $\p\in\Phi$.
Then $\p\notin\Ass I^i$ for all $0\le i\le n$.
Hence $\Gamma_{\p R_\p}(I^i_\p)=0$, and therefore $\H_{\p R_\p}^i(M_\p)=0$ for all $0\le i\le n$.

We claim that $\dim_{R_\p}M_\p \le n$.
Indeed, in case (i), the statement holds since $\dim R_\p = \height \p \le n$.
In case (ii), we have $\Supp_R M \subseteq \Supp_R I^0 \subseteq \cl(\Psi)$, where $\cl(\Psi)$ stands for the set of prime ideals $\p$ such that $\p\supseteq\q$ for some $\q\in\Psi$.
It is easy to see that $\dim_RM \le n$, and hence $\dim_{R_\p}M_\p \le n$.

It follows from Grothendieck's vanishing theorem \cite[6.1.2]{BS} that $\H_{\p R_\p}^{> n}(M_\p) = 0$.
Therefore, we get $\H_{\p R_\p}^i(M_\p)=0$ for all integers $i$, and $M_\p\in\supp^{-1}_{\Mod R_\p}(\{\p R_\p\}^\complement)$ by \cite[Remark 3.8(1)]{cd}.
We have $\p R_\p\notin\supp_{R_\p}M_\p$, which implies $\p\notin\supp_RM$.
Now we conclude that $\supp M\subseteq\Phi^\complement$.
Note that $N\cong\mho^{n+1}M\oplus J$ for some direct summand $J$ of $I^n$, which implies $\Ass N=\Ass\mho^{n+1}M\cup\Ass J$.
We have $\Ass\mho^{n+1}M=\Ass\E^{n+1}(M)\subseteq\supp M\subseteq\Phi^\complement$, while $\Ass J\subseteq\Ass I^n\subseteq\Phi^\complement$.
Hence $\Ass N\subseteq\Phi^\complement$, which is what we have wanted to deduce.
\end{proof}	

The following result refines \cite[Corollary A.5]{Kra08} to assert that in the case $\dim R\ge2$ there exists a coherent subset which is generalization-closed.
Also, for $n=1$ this theorem contains \cite[Corollary 4.3]{Kra08}.

\begin{cor}\label{10}
The following are equivalent for an integer $n\ge0$.
\begin{enumerate}[\rm(1)]
\item
Every subset of $\Spec R$ is $n$-coherent.
\item
Every generalization-closed subset of $\Spec R$ is $n$-coherent.
\item
The set $(\Max R)^\complement$ of nonmaximal prime ideals of $R$ is $n$-coherent.
\item
One has $\H_\Phi^{>n}(-)=0$ for every specialization-closed subset $\Phi$ of $\Spec R$.
\item
The inequality $\dim R \le n$ holds.
\end{enumerate}
\end{cor}

\begin{proof}
The equivalences (3) $\Leftrightarrow$ (4) $\Leftrightarrow$ (5) follow from \cite[Theorem 4.13]{cd}, Proposition \ref{9}(2) and Remark \ref{4}(4), while the implications (1) $\Rightarrow$ (2) $\Rightarrow$ (3) are obvious.
Let $\Xi$ be a subset of $\Spec R$.
Applying Proposition \ref{923-1}(i) to $\Phi:=\Xi^\complement$ and $\Psi:=\Spec R$, we see that (5) implies (1).
\end{proof}

\begin{rem}
Using the above result, we see that Proposition \ref{923-1} fails without (i) or (ii).
In fact, let $\Phi=\Max R$ and $\Psi=\Spec R$.
Then for any integer $n\ge0$ the subset $\Psi$ is $n$-coherent, but $\Psi\setminus\Phi$ is not if $\dim R>n$ by Corollary \ref{10}.
\end{rem}

We here record a remarkable statement.

\begin{cor}\label{infcoh}
Every subset $\Phi$ of $\Spec R$ is $\infty$-coherent.	
\end{cor}

\begin{proof}
Thanks to Proposition \ref{2}(2), we may reduce to the case where $R$ is a local ring.	
Then every subset of $\Spec R$ is $(\dim R)$-coherent by Corollary \ref{10}, and hence is $\infty$-coherent by Remark \ref{4}(2).
\end{proof}

Now we prove the following theorem.
In view of Remarks \ref{rem2}, \ref{4}, Proposition \ref{3} and \cite[Lemma 3.5]{Kra08}, one observes that this theorem for $n=0$ (resp. $n=1$) yields a one-to-one correspondence between the localizing (resp. wide and $\oplus$-closed) subcategories of $\Mod R$ and the specialization-closed (resp. coherent) subsets of $\Spec R$, which is nothing but the classification theorem of Gabriel \cite[Proposition VI.4]{Gab} (resp. Krause \cite[Theorem 3.1]{Kra08}).

\begin{thm}\label{7}
Let $n$ be a nonnegative integer or $\infty$.
The assignments $\X \mapsto \supp\X$ and $\supp_{\Mod R}^{-1}(\Phi)\mapsfrom\Phi$ give a bijective correspondence between the $n$-wide $(\E,\oplus)$-closed subcategories of $\Mod R$, and the $n$-coherent subsets of $\Spec R$.
\end{thm}

\begin{proof}
Let $n\in\NN$ (resp. $n=\infty$).
We prove that the map $\X \mapsto \supp\X$ is well-defined.
Let $\X$ be an $n$-wide $(\E,\oplus)$-closed subcategory of $\Mod R$.
Let $\cdots\to I_1\to I_0\to C\to0$ be an exact sequence with $I_i\in\Ass_{\Inj R}^{-1}(\supp\X)$ for all $0\le i\le n$ (resp. all $i\ge0$).
Then we have $\Ass I_i\subseteq\supp\X$ and find $X_i\in\X$ with $\Ass I_i\subseteq\supp X_i$ since $\X$ is closed under direct sums.
We observe that $I_i$ is a direct summand of a direct sum of copies of $\bigoplus_{j\ge0}\E^j(X_i)$.
As $\X$ is $(\E,\oplus)$-closed and it is also closed under direct summands by Remark \ref{rem2}(4), the module $I_i$ belongs to $\X$ for all $0\le i\le n$ (resp. all $i\ge0$).
The $n$-wideness of $\X$ shows $C\in\X$, which implies $\e_R(C)\in\Ass_{\Inj R}^{-1}(\supp\X)$.
We conclude that $\supp\X$ is $n$-coherent.

In view of \cite[Theorem 2.3]{hcls} and Remark \ref{rem2}(5), it is enough to prove that the map $\supp_{\Mod R}^{-1}(\Phi)\mapsfrom\Phi$ is well-defined, which follows from Proposition \ref{9} in the case $n\in\NN$.
Let $\Phi$ be an $\infty$-coherent subset of $\Spec R$.
For any prime ideal $\p$ of $R$, we see from Corollary \ref{10} and Proposition \ref{9}(2) that $\supp_{\Mod R_\p}^{-1}(\Phi_\p)$ is $(\height \p)$-wide, and it is $\infty$-wide by Remark \ref{rem2}(1).
Proposition \ref{222} deduces that $\supp_{\Mod R}^{-1}(\Phi)$ is $\infty$-wide.
\end{proof}

\begin{cor}\label{36'}
The assignments $\X \mapsto \supp\X$ and $\supp_{\Mod R}^{-1}(\Phi)\mapsfrom\Phi$ give a bijective correspondence between the localizing  (resp. $\oplus$-closed wide) subcategories of $\Mod R$ and the specialization-closed (resp. coherent) subsets of $\Spec R$.
\end{cor}

\begin{proof}
It follows from Remark \ref{rem2}(2)(3) that the $0$-wide (resp. $1$-wide) subcategories are the Serre (resp. wide) subcategories of $\Mod R$ and it follows from Proposition \ref{3} and Remark \ref{4}(3) that the $0$-coherent (resp. $1$-coherent) subsets are the specialization-closed (resp. coherent) subsets of $\Spec R$.
Therefore, for $n=0$ (resp. $n=1$), Theorem \ref{7} is directly translated as the bijective correspondence between $\E$-closed localizing (resp. $(\E, \oplus)$-closed wide) subcategories of $\Mod R$ and specialization-closed (resp. coherent) subsets of $\Spec R$.
Since \cite[Lemma 3.5]{Kra08} shows every $\oplus$-closed wide subcategory is $\E$-closed, this bijection is nothing but the one that we want.
\if0
\old{It follows from Remarks \ref{rem2}(2)(3), \ref{4}(3), Proposition \ref{3}, and \cite[Lemma 3.5]{Kra08} that the $0$-coherent (resp. $1$-coherent) subsets are the specialization-closed (resp. coherent) subsets and that the $0$-wide (resp. $1$-wide) $(\E, \oplus)$-closed subcategories are the localizing (resp. wide $\oplus$-closed) subcategories.
Thus, the statement follows by letting $n=0,1$ in Theorem \ref{7}.}
\fi
\end{proof}

\begin{ques}
As we have seen above, $\oplus$-closed $n$-wide subcategories of $\Mod R$ are automatically $\E$-closed for $n=0$ and $1$ by \cite[Lemma 3.5]{Kra08}.
However, its proof breaks down for $n>1$ and we do not have any examples of $\oplus$-closed $n$-wide subcategories of $\Mod R$ which are not $\E$-closed.
Can we show that all $\oplus$-closed $n$-wide subcategories of $\Mod R$ are $\E$-closed? 
\end{ques}

\if0
\begin{cor}\label{36}
The assignments $\X \mapsto \supp\X$ and $\supp_{\Mod R}^{-1}(\Phi)\mapsfrom\Phi$ give a bijective correspondence between the wide $\oplus$-closed subcategories of $\Mod R$ and the coherent subsets of $\Spec R$.
\end{cor}

\begin{proof}
The assertion follows from Remarks \ref{rem2}(3), \ref{4}(3), \cite[Lemma 3.5]{Kra08} and letting $n=1$ in Theorem \ref{7}.
\end{proof}
\fi

Takahashi \cite[Theorem 2.3]{hcls} gives a classification of the $(\ee,\oplus,\ominus)$-closed subcategories of $\Mod R$.
As an application of Theorem \ref{7}, we recover this classification by making a connection with $\infty$-wide subcategories.

\begin{cor}\label{35}
Let $\X$ be an $(\E, \oplus)$-closed subcategory of $\Mod R$.
Then $\X$ is $(\ee,\ominus)$-closed if and only if $\X$ is $\infty$-wide.
In particular, there are one-to-one correspondences among the $(\ee,\oplus,\ominus)$-closed subcategories of $\Mod R$, the $(\E, \oplus)$-closed $\infty$-wide subcategory of $\Mod R$, and the subsets of $\Spec R$.
\end{cor}

\begin{proof}
Suppose that $\X$ is $(\ee,\ominus)$-closed.
Then the equality $\X = \bigcap_{i\ge 0} \C_\X^i$ holds by Remark \ref{8}(4), and this means that $\X = \supp_{\Mod R}^{-1}(\Ass(\Inj R \cap \X))$ by \cite[Lemma 3.3]{Kra08}.
It follows from Corollary \ref{infcoh} and Theorem \ref{7} that $\X$ is $\infty$-wide.
Conversely, suppose that $\X$ is $\infty$-wide.
Then $\X$ is $\ominus$-closed by Remark \ref{rem2}(4).
For an $R$-module $M$ with $\E^i(M) \in \X$ for all $i \ge 0$, one has $M\in\X$ since $\X$ is closed under $\infty$-kernels.
This shows that $\X$ is $\ee$-closed.
Thus the proof of the first assertion is completed.
The second assertion follows from the first one and Theorem \ref{7}.
\end{proof}

We state another corollary of Theorem \ref{7}.
In general, it is difficult to check whether a given subset is $n$-coherent or not.
The first assertion of the corollary gives a necessary condition of $n$-coherence.
We mentioned in Remark \ref{4}(2) that $n$-coherence implies $(n+1)$-coherence for each $n \ge 0$.
The second assertion of the corollary shows that this implication can be strict for an arbitrary $n$.

\begin{cor}\label{33}
\begin{enumerate}[\rm(1)]
\item
Let $\a$ be an ideal of $R$ and $n\ge 0$ an integer.
Let $M$ be a finitely generated $R$-module with $\a M\ne M$.
If $\d(\a)$ is an $n$-coherent subset of $\Spec R$, then one has $\grade(\a,M) \le n$.
\item
Let $\xx= x_1, \ldots, x_n$ be a sequence of elements of $R$.
Then $\d(\xx)$ is an $n$-coherent subset of $\Spec R$.
If the sequence $\xx$ is $R$-regular, then $\d(\xx)$ is not $(n-1)$-coherent.
\item
Let $\Phi$ be a specialization-closed subset of $\Spec R$.
Then the following equivalences hold.
$$
\H_\Phi^{>0}(M)=0\text{ for all $M\in\Mod R$} 
\ \Leftrightarrow\ \supp_{\Mod R}^{-1}(\Phi^\complement) \text{ is localizing}
\ \Leftrightarrow\ \Phi^\complement \text{ is specialization-closed}
\ \Leftrightarrow\ \Phi \text{ is clopen}.
$$
\end{enumerate}
\end{cor}

\begin{proof}
Assertions (1) and (2) immediately follow from Theorem \ref{7} and \cite[Proposition 4.10 and Example 4.11]{cd}.
As for assertion (3), the first and second equivalences are direct consequences of \cite[Theorem 4.9(1)]{cd} and Theorem \ref{7}.

Assume $\Phi$ is clopen and write $\Phi = \V(\a)$ and $\Phi^\complement = \V(\b)$ with radical ideals $\a,\b$.
Then $\a + \b = R$ and $\a \cap \b = \sqrt{0}$.
Hence $\H_{\a+\b}^{\ge0}(M)=\H_{\a\cap\b}^{>0}(M)=0$ for any $M\in\Mod R$.
The Mayer--Vietoris sequence \cite[3.2.3]{BS} implies $\H_\a^{>0}(M)=0$.

If $\Phi^\complement$ is specialization-closed, then $\Phi$ is both specialization-closed and generalization-closed and so is $\Phi^\complement$.
By symmetry, it is enough to show that $\Phi$ is closed.
As $\Phi$ is specialization-closed, we can write $\Phi = \bigcup_{\p \in \min \Phi} \V(\p)$, where $\min \Phi$ stands for the set of minimal elements of $\Phi$ with respect to the inclusion relation.
As $\Phi$ is generalization-closed, $\min \Phi$ consists only of minimal prime ideals and hence is a finite set.
We see that $\Phi$ is closed.
\end{proof}

%%%%%%%%%%%%%%%%%%%%%%%%%%%%%%%%%%%%%%%%%%%%%%
\section{Classification of $n$-uniform localizing subcategories of $\D(\Mod R)$}

In the last section, we gave a correspondence between the $n$-coherent subsets of $\Spec R$ and the $(\E, \oplus)$-closed $n$-wide subcategories of $\Mod R$.
In this section, we consider the derived category analogue of this correspondence.
To this end, we introduce two kinds of subcategories of $\D(\Mod R)$.

\begin{dfn}\label{5.1}
Let $n\in\NN\cup\{\infty\}$ .
\begin{enumerate}[\rm(1)]
\item
We say that a subcategory $\X$ of $\D(\Mod R)$ is {\em $n$-uniform} provided that if $X\in\X$ and $i\in\Z$ satisfy $\H^jX=0$ for all $i\ne j\in(i-n,i+n)$ and $X^j \in \X$ for all integers $j \in [i-n, i+n]$, then $\ZZ^iX,\,X^i/\BB^iX\in\X$.
\item
We say that a subcategory $\X$ of $\D(\Mod R)$ is {\em $n$-consistent} provided that if $X\in\X$ and $i\in\Z$ satisfy $\H^jX=0$ for all $i\ne j\in(i-n,i+n)$, then $\H^iX\in\X$.
\end{enumerate}
\end{dfn}

\begin{rem}\label{28}
\begin{enumerate}[(1)]
\item
A subcategory $\X$ of $\D(\Mod R)$ is $0$-uniform if and only if for each $X \in \X$ and $i\in\Z$ with $X^i \in \X$ one has $\ZZ^i X,\,X^i/\BB^iX\in\X$.
\item
A subcategory of $\D(\Mod R)$ is $0$-consistent, if and only if it is $1$-consistent, if and only if it is closed under cohomologies.
\item
Let $n$ be a nonnegative integer.
Then any $n$-uniform (resp. $n$-consistent) subcategory of $\D(\Mod R)$ is both $(n+1)$-uniform and $\infty$-uniform (resp. $(n+1)$-consistent and $\infty$-consistent).
\item
Every thick subcategory of $\D(\Mod R)$ is $\infty$-consistent.
In fact, let $X\in\X$, $i\in\Z$ and $\H^jX=0$ for all $i\ne j\in(-\infty,\infty)$.
Then the complex $X$ is isomorphic to the stalk complex $\H^iX[-i]$ in $\D(\Mod R)$.
Since $X\in\X$, we have $\H^iX\in\X$.
\end{enumerate}
\end{rem}

The condition in the definition of $n$-consistency is closed under quasi-isomorphisms, but the condition in the definition of $n$-uniformity is not.
Also, $n$-consistency looks simpler than $n$-uniformity.
We shall show in Theorem \ref{24} that, whenever $n\ne0$, these two notions are equivalent for localizing subcategories\footnote{Thus the reader may wonder why we introduce the notion of $n$-uniform subcategories. A benefit from introducing this notion is that Theorem \ref{8-2} can be shown by combination of classical results. More precisely, it can be done without Lemma \ref{23} which comes from the Nakamura--Yoshino theory, and so even the reader who does not know this theory can understand it.}.
We first investigate $0$-uniform localizing subcategories in relation to smashing subcategories.

\begin{dfn}\label{5.3}
A thick subcategory $\X$ of $\D(\Mod R)$ is called {\em smashing} if the inclusion functor $\X\hookrightarrow\D(\Mod R)$ admits a right adjoint which preserves direct sums.
In particular, it is localizing; see Remark \ref{26}.
\end{dfn}

\begin{prop}\label{27}
The $0$-uniform localizing subcategories of $\D(\Mod R)$ are the smashing subcategories of $\D(\Mod R)$.
\end{prop}

\begin{proof}
Let $\X$ be a localizing subcategory of $\D(\Mod R)$.
Then $\X=\supp_{\D(\Mod R)}^{-1}(\Phi)$ for some subset $\Phi$ of $\Spec R$ by \cite[Theorem 2.8]{Nee}.
If $\X$ is smashing, then $\Phi$ is specialization-closed by \cite[Theorem 3.3]{Nee}, and $\X\cap\Mod R=\supp_{\Mod R}^{-1}(\Phi)$ is localizing by \cite[page 425]{Gab} (see also \cite[Corollary 3.6]{Kra08}).
Conversely, if $\X\cap\Mod R$ is localizing, then by \cite[page 425]{Gab} again $\supp\X=\supp(\X\cap\Mod R)=\Phi$ is specialization-closed, and $\X$ is smashing by \cite[Theorem 3.3]{Nee} again.
Thus it suffices to show that $\X$ is $0$-uniform if and only if $\X\cap\Mod R$ is localizing.

Suppose that the subcategory $\X$ is $0$-uniform.
The localizing property of $\X$ implies that $\X\cap\Mod R$ is closed under direct sums.
Let $0\to L\to M\to N\to0$ be an exact sequence of $R$-modules.
Then an exact triangle $L\to M\to N\to L[1]$ in $\D(\Mod R)$ is induced.
If $L,N\in\X$, then $M\in\X$ since $\X$ is localizing.
Now, assume $M\in\X$.
Then consider the complex $X=(0\to L\to M\to N\to 0)$ with $M$ being in degree $0$.
This complex is exact, so that $X\cong0$ in $\D(\Mod R)$.
As $\X$ contains $0$, it also contains $X$.
Since $X^0=M$ belongs to $\X$, the modules $L=\ZZ^0X$ and $N=X^0/\BB^0X$ belong to $\X$ as well.
We conclude that $\X\cap\Mod R$ is localizing.

Conversely, suppose that $\X\cap\Mod R$ is a localizing subcategory of $\Mod R$.
Let $X\in\X$, $i\in\Z$ and $X^i\in\X$.
Then $X^i$ belongs to $\X\cap\Mod R$.
The natural injection $\ZZ^iX\hookrightarrow X^i$ and surjection $X^i\twoheadrightarrow X^i/\BB^iX$ show that $\ZZ^iX$ and $X^i/\BB^iX$ are in $\X\cap\Mod R$.
It follows that $\X$ is $0$-uniform.
\end{proof}

The following theorem is one of the main results in this section, which gives a derived category analogue of the correspondence \ref{7} using $n$-uniform subcategories.

\begin{thm}\label{8-2}
Let $n\in\NN\cup\{\infty\}$.
The assignments $\X\mapsto\supp\X$ and $\supp^{-1}_{\D(\Mod R)}(\Phi)\mapsfrom\Phi$ give a bijective correspondence between the $n$-uniform localizing subcategories of $\D(\Mod R)$ and the $n$-coherent subsets of $\Spec R$.
\end{thm}

\begin{proof}
In view of \cite[Theorem 2.8]{Nee}, it suffices to show that (a) $\supp\X$ is $n$-coherent for an $n$-uniform localizing subcategory $\X$ of $\D(\Mod R)$ and that (b) $\supp^{-1}_{\D(\Mod R)}(\Phi)$ is $n$-uniform for an $n$-coherent subset $\Phi$ of $\Spec R$.

(a) 
The case $n=0$ is settled since the $0$-uniform subcategories are the localizing subcategories of $\D(\Mod R)$ whose supports are specialization-closed by Remark \ref{27} and \cite[Theorem 3.3]{Nee}, while Corollary \ref{infcoh} shows the case $n=\infty$.
So, let $0<n<\infty$.
Put $\Phi=\supp\X$.
We have $\X=\supp_{\D(\Mod R)}^{-1}(\Phi)$ by \cite[Theorem 2.8]{Nee}, and $\X\cap\Mod R=\supp_{\Mod R}^{-1}(\Phi)$.
By \cite[Theorem 2.3]{hcls} we get $\supp(\X\cap\Mod R)=\Phi$.
In view of Theorem \ref{7}, it suffices to show that $\X\cap\Mod R$ is $n$-wide and $(\E,\oplus)$-closed.
It follows from \cite[Theorem 2.3]{hcls} again that $\X\cap\Mod R$ is $(\E,\oplus)$-closed.
As $\X$ is localizing, we easily see that $\X\cap\Mod R$ is closed under extensions.
It remains to verify that $\X\cap\Mod R$ is closed under $n$-kernels and $n$-cokernels.
Let $0\to K\to X^0\to\cdots\to X^n\to C\to0$ be an exact sequence of $R$-modules with $X^i\in\X$ for all $0\le i\le n$.
Then consider the complex $X=(\cdots\to0\to X^0\to\cdots\to X^n\to0\to\cdots)$.
For each $1\le i\le n$ there is an exact triangle $X^i[-i]\to C_i\to C_{i-1}\to X^i[-i+1]$ with $C_n=X$ and $C_0=X^0$.
We inductively see that $X$ belongs to $\X$.
Since $\H^jX=0$ for all $0\ne j\in(-n,n)$ and $n\ne j\in(0,2n)$, the modules $K=\ZZ^0X$ and $C=X^n/\BB^nX$ belong to $\X$.
Therefore, $\X\cap\Mod R$ is closed under $n$-kernels and $n$-cokernels.

(b) Put $\X=\supp^{-1}_{\D(\Mod R)}(\Phi)$.
Let $X\in\X$, $i\in\Z$, $\H^jX=0$ for all integers $j$ with $i\ne j\in(i-n,i+n)$, and $X^j \in \X$ for all integers $j \in [i-n, i+n]$.
Let $n\in\NN$ (resp. $n=\infty$).
There are two exact sequences $0\to\ZZ^iX\to X^i \to\cdots\to X^{i+n}$ and $X^{i-n}\to\cdots\to X^i\to X^i/\BB^iX\to0$ (resp. $0\to\ZZ^iX\to X^i\to X^{i+1}\to\cdots$ and $\cdots\to X^{i-1}\to X^i\to X^i/\BB^iX\to0$).
Theorem \ref{7} says that $\X\cap\Mod R=\supp^{-1}_{\Mod R}(\Phi)$ is $n$-wide.
In particular, it is closed under $n$-kernels and $n$-cokernels.
The exact sequences show that $\ZZ^iX$ and $X^i/\BB^iX$ are in $\X\cap\Mod R$, and hence they belong to $\X$.
Thus, $\X$ is $n$-uniform.
\end{proof}

Corollary \ref{infcoh}, Theorem \ref{8-2} and \cite[Theorem 2.8]{Nee} imply the following.

\begin{cor}
Every localizing subcategory of $\D(\Mod R)$ is $\infty$-uniform.
\end{cor}

Next, we consider classifying $n$-consistent subcategories of $\D(\Mod R)$.
We shall do it by showing that $n$-uniformity and $n$-consistency are equivalent for localizing subcategories.
To show that $n$-uniformity implies $n$-consistency, we use the lemma below.
This follows from the dual version of \cite[Corollary 7.14]{NY}, whose establishment is guaranteed by \cite[the three lines at the end of Section 7]{NY}.

\begin{lem}[Nakamura--Yoshino]\label{23}
Suppose $\dim R<\infty$.
Let $\Phi$ be a subset of $\Spec R$, and set $\X=\supp^{-1}_{\D(\Mod R)}(\Phi)$.
Then for every complex $X\in\X$ there exists a complex $Y\in\D(\Mod R)$ such that $Y\cong X$ and $Y^i\in\X$ for all $i\in\Z$.
\end{lem}

\begin{proof}[Sketch of Proof]
One can verify this lemma inductively via \cite[Definition 7.6]{NY} by combining \cite[Theorem 3.12]{NY0}, \cite[Lemma 9.1(2)]{NY}, \cite[Lemma 3.5.1]{Lip} and \cite[Theorem 3.22]{NY} (or \cite[Theorem 3.13]{NY0}), see also \cite[Corollary 3.4]{NY0}.

Concretely, we can construct such a complex $Y$ as follows.
Let $\mathbb{W} = \{W_i\}_{i=1}^n$ be a {\it system of slices} of $\Phi$ in the sense of  \cite[Definition 7.6]{NY} (for example, we can take $W_i := \{\p \in \Phi \mid \dim R/\p =i\}$).
We then replace the functor $\overline{\lambda}^{(i_m, \ldots, i_0)}$ and the complex of functors $L^{\mathbb{W}}$ in \cite{NY} with their categorical duals
$$
\underline{\gamma}_{(i_m, \ldots, i_0)} := \underline{\gamma}_{W_{i_m}} \cdots \underline{\gamma}_{W_{i_0}},\qquad\underline{\gamma}_{W_i}:= \bigoplus_{\p \in W_i} \Gamma_{\V(\p)} \Hom_R(R_\p, - )\\
$$
and 
$$
R_{\mathbb{W}} := \left( 0 \to \underline{\gamma}_{(d, \ldots, 0)} \to \bigoplus_{i_0 < \cdots < i_{d-1}} \underline{\gamma}_{(i_{d-1}, \ldots, i_0)} \to \cdots \to \bigoplus_{i_0 < i_1} \underline{\gamma}_{(i_{1}, i_0)} \to \cdots \to \bigoplus_{i} \underline{\gamma}_{(i)} \to 0 \right)
$$
respectively. 
Here the differentials of $R_{\mathbb{W}}$ are induced from the canonical natural transformations $\underline{\gamma}_{W_i} \to \mathrm{id}_{\Mod R}$. 
This complex can be constructed inductively as follows.
Let $\mathbb{W}_r := \{W_i\}_{i=0}^r$ and denote by $\widetilde{R_{\mathbb{W}_r}}$ the augmentation of $R_{\mathbb{W}_r}$.
Then, we see that $\widetilde{R_{\mathbb{W}_{r+1}}}$ is the mapping cone of $\underline{\gamma}_{W_{r+1}} \widetilde{R_{\mathbb{W}_{r}}} \to \widetilde{R_{\mathbb{W}_{r}}}$ and hence its de-augmentation is $R_{\mathbb{W}_{r+1}}$.
Therefore, by induction, one can show that there is a natural isomorphism $\gamma_\Phi (M) \cong R_{\mathbb{W}} (M)$ in $\D(\Mod R)$ for any injective $R$-module $M$.
Here, the induction basis follows from \cite[Theorem 3.12]{NY0}.
For a complex $X \in \D(\Mod R)$ with a semi-injective resolution $I$, we get isomorphisms $X \cong \gamma_{\Phi}(I) \cong \mathrm{Tot}(R_{\mathbb{W}}(I))=:Y$ in $\D(\Mod R)$ such that each component of $Y$ belongs to $\supp_{\D(\Mod R)}^{-1}(\Phi)$, see \cite[Lemma 3.5.1]{Lip} and \cite[Theorem 3.22]{NY}.
\end{proof}

\begin{prop}\label{8-4}
Let $n \in \NN \cup \{\infty\}$.
\begin{enumerate}[\rm(1)]
\item
Assume $n\ne0$.
For an $n$-consistent localizing subcategory $\X$ of $\D(\Mod R)$, the subset $\supp \X$ of $\Spec R$ is $n$-coherent.
\item
For an $n$-coherent subset $\Phi$ of $\Spec R$, the subcategory $\supp^{-1}_{\D(\Mod R)}(\Phi)$ of $\D(\Mod R)$ is $n$-consistent. 
\end{enumerate}
\end{prop}

\begin{proof}
(1) The assertion is shown by the same arguments as in (a) in the proof of Theorem \ref{8-2} for $n\ne0$, where using $n$-uniformity at the end is just replaced with using $n$-consistency.

(2) Set $\X= \supp^{-1}_{\D(\Mod R)}(\Phi)$.
Let $X \in \X$ and $i \in \Z$ be such that $\H^jX = 0$ for all $i \neq j \in (i-n, i+n)$.
We prove $\H^iX \in \X$ in two steps: (i) the case where $\dim R < \infty$ and (ii) the general case.

(i) By Lemma \ref{23}, we may assume $X^j \in \X$ for all $j \in \Z$.
Since $\X$ is $n$-uniform by Theorem \ref{8-2}, we have $X^i/\BB^iX \in \X$.
Note by Theorem \ref{7} that $\X\cap\Mod R=\supp^{-1}_{\Mod R}(\Phi)$ is $n$-wide, and in particular, it is closed under $n$-kernels.
From the exact sequence $0\to\H^iX\to X^i/\BB^iX \to X^{i+1}\to\cdots\to X^{i+n}$, we obtain $\H^iX\in\X\cap\Mod R\subseteq\X$.

(ii) Fix a prime ideal $\p$ of $R$.
Since $R_\p$ has finite Krull dimension, it follows from (i) and Proposition \ref{2}(2) that $\Y:=\supp_{\D(\Mod R_\p)}^{-1}(\Phi_\p)$ is $n$-consistent.
As $\kappa(\q R_\p)\ltensor_{R_\p}X_\p\cong\kappa(\q)\ltensor_RX$ for $\q\in\Spec R$ with $\q\subseteq\p$, we observe that $X_\p$ belongs to $\Y$.
Also, $\H^j(X_\p)=(\H^jX)_\p=0$ for all integers $j$ with $i\ne j\in(i-n,i+n)$.
Hence $\H^i(X_\p)\in\Y$, and therefore $\supp_{R_\p}(\H^iX)_\p\subseteq\Phi_\p$ for all $\p\in\Spec R$.
It is seen that $\H^iX\in\X$.
\end{proof}

As we have promised, we give the following classification theorem of $n$-consistent localizing subcategories.

\begin{thm}\label{24}
Let $0<n\in\NN\cup\{\infty\}$.
Then a localizing subcategory of $\D(\Mod R)$ is $n$-uniform if and only if it is $n$-consistent.
In particular, the assignments $\X\mapsto\supp\X$ and $\supp^{-1}_{\D(\Mod R)}(\Phi)\mapsfrom\Phi$ give a bijective correspondence between the $n$-consistent localizing subcategories of $\D(\Mod R)$ and the $n$-coherent subsets of $\Spec R$.
\end{thm}

\begin{proof}
We begin with showing the first assertion.
According to Theorem \ref{8-2} and Proposition \ref{8-4}(2), it is enough to prove that every $n$-consistent subcategory $\X$ of $\D(\Mod R)$ is $n$-uniform.
Proposition \ref{8-4}(1) implies that $\supp \X$ is $n$-coherent, and hence $\supp_{\D(\Mod R)}^{-1}(\supp \X)$ is $n$-uniform by Theorem \ref{8-2}.
We have $\X=\supp_{\D(\Mod R)}^{-1}(\supp \X)$ by \cite[Theorem 2.8]{Nee}.
It follows that $\X$ is $n$-uniform, as desired.
The second assertion follows from the first and Theorem \ref{8-2}.
\end{proof}

Letting $n=0$ in Theorem \ref{8-2} and $n=1, \infty$ in Theorem \ref{24} respectively and using Proposition \ref{27} and Remark \ref{28}(2), we have the following corollary.
These bijections are given in \cite[Theorem 3.3]{Nee} and \cite[Theorem 5.2]{Kra08} (see also \cite[Main Theorem]{hcls}).

\begin{cor}\label{25}
The assignments $\X\mapsto\supp\X$ and $\supp^{-1}_{\D(\Mod R)}(\Phi)\mapsfrom\Phi$ give one-to-one correspondences:

\begin{itemize}
\item
$\{\text{localizing subcategories of $\D(\Mod R)$}\}\leftrightarrow\{\text{subsets of $\Spec R$}\}$,

\item
$\{\text{cohomology-closed localizing subcategories of $\D(\Mod R)$}\}\leftrightarrow\{\text{coherent subsets of $\Spec R$}\}$.

\item
$\{\text{smashing subcategories of $\D(\Mod R)$}\}\leftrightarrow\{\text{specialization-closed subsets of $\Spec R$}\}$,
\end{itemize}
\end{cor}

\begin{rem}
It is clarified by \cite[Remark 2.3]{CI} that \cite[Proposition 5.1]{Kra08} does not necessarily hold true, namely, for an unbounded complex $X$ with minimal semi-injective resolution $I$, the equality
$$
\supp X = \bigcup_{i \in \Z} \Ass I^i
$$ 
may fail\footnote{This does not affect \cite[Lemma 3.3]{Kra08}, \cite[Lemma 3.5]{Kra08} or \cite[Theorem 2.3]{hcls}}.
However, Corollary \ref{25} guarantees that the assertion of \cite[Theorem 5.2]{Kra08} is correct.
Indeed, as a direct consequence of Corollary \ref{25}, for a subset $\Phi$ of $\Spec R$, $\supp^{-1}_{\D(\Mod R)} (\Phi)$ is closed under taking cohomologies if and only if $\Phi$ is coherent.
%\footnote{It is asserted in \cite[Theorem 5.2]{Kra08} that for a coherent subset $\Phi$ of $\Spec R$ and an $R$-complex $X$ one has $X\in\supp^{-1}_{\D(\Mod R)}(\Phi)$ if and only if $\H^iX\in\supp^{-1}_{\D(\Mod R)}(\Phi)$ for all $i\in\Z$. 
%Corollary \ref{25} guarantees the correctness of the ``only if'' part, while the proof of the ``if part'' in \cite[Theorem 5.2]{Kra08} remains valid.
%By the way, before the project of the present paper started, Tsutomu Nakamura had told the second-named author (Takahashi) that the assertion of \cite[Theorem 5.2]{Kra08} is correct.}.
Here we give a proof of \cite[Proposition 2.3(2)]{hcls}, i.e., the equality
$$
\supp^{-1}_{\D(\Mod R)}(\supp \M) = \overline{\M} := \{X \in \D(\Mod R) \mid \H^i(X) \in \M \,\,\, (\forall i \in \Z) \} 
$$
for a $\oplus$-closed wide subcategory $\M$ of $\Mod R$ without using \cite[Proposition 5.1]{Kra08}.

Fix such a subcategory $\M$ of $\Mod R$.
We note that the proof of the inclusion $\overline{\M} \subseteq \supp_{\D(\Mod R)}^{-1}(\supp\M)$ in \cite[Proposition 2.3(2)]{hcls} just uses \cite[Theorem 5.2]{Kra08}.
By contrast, the converse inclusion uses \cite[Proposition 5.1]{Kra08}.
Thus, we need to check the inclusion $\overline{\M} \supseteq \supp_{\D(\Mod R)}^{-1}(\supp\M)$.
Let $\widetilde\M$ be the smallest localizing subcategory of $\D(\Mod R)$ containing $\M$.
Then we have
$$
\supp_{\D(\Mod R)}^{-1}(\supp\M)
\subseteq\supp_{\D(\Mod R)}^{-1}(\supp\widetilde\M)
=\widetilde\M \subseteq \overline{\M}.
$$
The first inclusion is trivial since $\widetilde{\M}$ contains $\M$ and the equality follows from \cite[Theorem 2.8]{Nee}.
It is easy to see that $\overline\M$ is localizing, and contains $\M$.
The last inclusion follows from this.
Therefore, we obtain the inclusion $\overline{\M} \supseteq \supp_{\D(\Mod R)}^{-1}(\supp\M)$.
Thus the proof of \cite[Proposition 2.3(2)]{hcls} is completed.

\end{rem}

%%%%%%%%%%%%%%%%%%%%%%%%%%%%%%%%%%%%%%%%%%%%%%
\section{Several classifications of subcategories by restriction}

This section concerns some restrictions of classifications obtained in the previous two sections.
First of all, we consider restriction to $\Pi$-closed subcategories of $\Mod R$ and $\D(\Mod R)$.

\begin{prop}\label{15}
The following are equivalent for a subset $\Phi$ of $\Spec R$.
\begin{enumerate}[\rm(1)]
\item
The subset $\Phi$ is generalization-closed.
\item
The subcategory $\supp_{\Mod R}^{-1} (\Phi)$ is closed under direct products.
\item
The subcategory $\supp_{\Mod R}^{-1} (\Phi)$ is closed under countable direct products.
\item
The subcategory $\supp_{\D(\Mod R)}^{-1} (\Phi)$ is closed under direct products.
\item
The subcategory $\supp_{\D(\Mod R)}^{-1} (\Phi)$ is closed under countable direct products.
\end{enumerate}
\end{prop}

\begin{proof}
We prove that (1) implies (2).
Let $\{M_\lambda\}_{\lambda \in \Lambda}$ be a family of modules in $\supp_{\Mod R}^{-1} (\Phi)$.
The product $\Pi_{\lambda \in \Lambda} \E(M_\lambda)$ of complexes is an injective resolution of the module $\Pi_{\lambda \in \Lambda} M_\lambda$.
We may assume that each $M_\lambda$ is injective to prove $\Ass(\Pi_{\lambda \in \Lambda} M_\lambda)\subseteq \Phi$.
Take $\p \in \Ass(\Pi_{\lambda \in \Lambda} M_\lambda)$.
Then $\p = \ann (m_\lambda)_{\lambda \in \Lambda}$ for some $0 \neq (m_\lambda)_{\lambda \in \Lambda} \in \Pi_{\lambda \in \Lambda} M_\lambda$.
There is $\lambda\in\Lambda$ with $m_\lambda\ne0$, and there exists $\q\in\Ass M_\lambda$ such that $\ann m_\lambda \subseteq \q$ (see \cite[Theorem 6.1]{M}).
Note that $\p\subseteq\ann m_\lambda$ and $\Ass M_\lambda\subseteq\Phi$.
As $\Phi$ is generalization-closed, it follows that $\p \in \Phi$.

We prove that (3) implies (1).
Take an inclusion $\p \subseteq \q$ of prime ideals with $\q \in \Phi$.
There is a sequence
\begin{equation}\label{16}
\textstyle
R/\p \hookrightarrow R_\q / \p R_\q \hookrightarrow \Pi_{r \ge 1} R_\q /(\q^r + \p)R_\q \hookrightarrow \Pi_{r \ge 1} \e_R(R /(\q^r + \p))_\q
\end{equation}
of injective maps, where the second injection follows from Krull's intersection theorem.
Since the ideal $(\q^r + \p)R_\q$ is $\q R_\q$-primary, we have $\Ass_{R_\q}(\e_R(R_\q /(\q^r + \p))_\q)= \{\q R_\q\}$, and hence $\Ass_{R}(\e_R(R /(\q^r + \p))_\q)= \{\q \}$ by \cite[Theorem 6.2]{M}.
Note that $\e_R(R /(\q^r + \p))_\q$ is injective as an $R$-module.
The module $\e_R(R /(\q^r + \p))_\q$ is in $\supp_{\Mod R}^{-1} (\Phi)$.
By assumption, the product $\Pi_{r \ge 1} \e_R(R /(\q^r + \p))_\q$ is also in $\supp_{\Mod R}^{-1} (\Phi)$, and hence $\Ass_R(\Pi_{r \ge 1} \e_R(R /(\q^r + \p))_\q)\subseteq \Phi$.
It follows from \eqref{16} that $\p\in\Phi$, and thus $\Phi$ is generalization-closed.

It is obvious that the implications (3) $\Leftarrow$ (2) $\Leftarrow$ (4) $\Rightarrow$ (5) $\Rightarrow$ (3) hold.
The proof of the proposition is done once we verify that (1) implies (4), which is a consequence of \cite[Lemma 4.5]{Nee11}.
\end{proof}

\begin{rem}\label{sm}
The above proposition gives an alternative proof of the smashing conjecture \cite[Theorem 3.3]{Nee} for $\D(\Mod R)$, that is, if the inclusion $\X \hookrightarrow \D(\Mod R)$ from a localizing subcategory of $\D(\Mod R)$ admits a coproduct preserving right adjoint, then $\X$ is generated by compact objects.	
Indeed, let $\X$ be such a subcategory of $\D(\Mod R)$.
Let us show that $\supp \X$ is specialization-closed.
Let $\Y$ be the kernel of $F$.
Then $\Y$ is closed under direct products as $F$ is a right adjoint functor, while $\Y$ is closed under direct sums as $F$ preserves direct sums.
It is easy to check that the equaity ${}^\perp\Y = \X$ holds.
By \cite[Theorem 2.8]{Nee} and Proposition \ref{15}, we have $\X = \supp^{-1}_{\D(\Mod R)}(\Phi)$ and $\Y = \supp^{-1}_{\D(\Mod R)}(\Psi)$ for some subset $\Phi$ and some generalization-closed subset $\Psi$ of $\Spec R$.
Here, recall that $\Hom_R(\e_R(R/\p), \e_R(E/\q)) = 0$ if and only if $\p \not\subseteq \q$ for $\p,\q\in\Spec R$; see \cite[Theorem 3.3.8(5)]{EJ}. 
This fact and the equality ${}^\perp\Y = \X$ show $\Phi = \Psi^\complement$, which is specialization-closed. 
\end{rem}

Now, we give classifications of $\Pi$-closed subcategories of $\Mod R$ and $\D(\Mod R)$.
The second classification can be considered as a dual of \cite[Theorem 3.3]{Nee}.

\begin{thm}\label{31}
Let $n\in\NN$.
\begin{enumerate}[\rm(1)]
\item
The assignments $\X \mapsto \supp \X$ and $\Phi \mapsto \supp_{\Mod R}^{-1}(\Phi)$ give a one-to-one correspondence
$$\textstyle
\{\text{$(\ee,\oplus,\ominus,\Pi)$-closed subcategories of $\Mod R$}\}\leftrightarrow
\{\text{generalization-closed subsets of $\Spec R$}\},
$$
which restricts to a one-to-one correspondence
$$\textstyle
\{\text{$n$-wide $(\E, \oplus,\Pi)$-closed subcategories of $\Mod R$}\}\leftrightarrow
\{\text{$n$-coherent generalization-closed subsets of $\Spec R$}\}.
$$
\item
The assignments $\X \mapsto \supp \X$ and $\Phi \mapsto \supp_{\D(\Mod R)}^{-1}(\Phi)$ give a one-to-one correspondence
$$\textstyle
\{\text{bilocalizing subcategories of $\D(\Mod R)$}\}\leftrightarrow
\{\text{generalization-closed subsets of $\Spec R$}\},
$$
which restricts to a one-to-one correspondence
$$\textstyle
\{\text{$n$-uniform bilocalizing subcategories of $\D(\Mod R)$}\}\leftrightarrow
\{\text{$n$-coherent generalization-closed subsets of $\Spec R$}\}.
$$
\end{enumerate}
\end{thm}

\begin{proof}
The combination of Theorems \ref{7}, \ref{8-2} and Proposition \ref{15} shows the two statements of the theorem.
\end{proof}

We study more about the above classification in the cases $n = 0, 1$.
Here, let us recall two notions from category theory.

\begin{dfn}\label{6.3}
\begin{enumerate}[(1)]
\item
Define a {\em bismashing} subcategory $\X$ of $\D(\Mod R)$ as a bilocalizing subcategory such that the inclusion functor $\X\hookrightarrow\D(\Mod R)$ has a right adjoint preserving direct sums and a left adjoint preserving direct products.
\item
Recall that a subcategory $\X$ of $\Mod R$ is called {\em bireflective} (resp. {\em Giraud}) if the inclusion functor $\X\hookrightarrow\Mod R$ admits left and right adjoints (resp. admits an exact left adjoint).
\end{enumerate}
\end{dfn}

\begin{rem}\label{32}
A wide $(\oplus,\Pi)$-closed subcategory of $\Mod R$ is nothing but a bireflective Giraud subcategory of $\Mod R$.
In fact, it follows from \cite[Proposition 3.8]{GL}, \cite[Theorem 1.2]{GP} and \cite[Theorem 4.8]{Sch} (see also \cite[Definition 5.2 and Theorem 5.3]{A}) that the wide $(\oplus,\Pi)$-closed subcategories $\X$ of $\Mod R$ are the bireflective extension-closed subcategories $\X$ of $\Mod R$, which bijectively correspond to the epiclasses of ring epimorphisms $R\to S$ with $\Tor_1^R(S,S)=0$ such that $\X$ is the image of the induced fully faithful functor $\Mod S\hookrightarrow\Mod R$, which are the epiclasses of flat ring epimorphisms $R\to S$ by \cite[Proposition 5.4]{epi}.
Let $\phi:R\to S$ be a ring epimorphism and $\theta:\Mod S\hookrightarrow\Mod R$ the induced functor.
If $\phi$ is flat, then the functor $-\otimes_RS:\Mod R\to\Mod S$ is an exact left adjoint to $\theta$.
Conversely, if $\theta$ has an exact left adjoint, then the uniqueness of (left) adjoints implies $\theta\cong-\otimes_RS$, and hence $\phi$ is flat.
\end{rem}

Now we can state and prove the following corollary.
The first and second assertions come from the cases $n=0$ and $n=1$ in Theorem \ref{31}, respectively.
The second assertion is (essentially) shown in \cite[Corollary 4.10]{epi}.

\begin{cor}\label{34}
There are one-to-one correspondences
\begin{enumerate}[\rm(1)]
\item
$\xymatrix@C4pc{
{\left\{
\begin{matrix}
\text{bilocalizing}\\
\text{subcategories of $\Mod R$}
\end{matrix}
\right\}}
\ar@<2pt>[r]^-\supp &
{\left\{
\begin{matrix}
\text{clopen subsets}\\
\text{of $\Spec R$}
\end{matrix}
\right\}}
\ar@<2pt>[l]^-{\supp^{-1}_{\Mod R}}\ar@<-2pt>[r]_-{\supp^{-1}_{\D(\Mod R)}} &
{\left\{
\begin{matrix}
\text{bismashing subcategories}\\
\text{of $\D(\Mod R)$}
\end{matrix}
\right\}}.
\ar@<-2pt>[l]_-\supp}$
\item
$\xymatrix@C4pc{
{\left\{
\begin{matrix}
\text{bireflective Giraud}\\
\text{subcategories of $\Mod R$}
\end{matrix}
\right\}}
\ar@<2pt>[r]^-\supp &
{\left\{
\begin{matrix}
\text{generalization-closed}\\
\text{coherent subsets of $\Spec R$}
\end{matrix}
\right\}}
\ar@<2pt>[l]^-{\supp^{-1}_{\Mod R}}\ar@<-2pt>[r]_-{\supp^{-1}_{\D(\Mod R)}} &
{\left\{
\begin{matrix}
\text{cohomology-closed bilocalizing}\\
\text{subcategories of $\D(\Mod R)$}
\end{matrix}
\right\}}.
\ar@<-2pt>[l]_-\supp}$
\end{enumerate}
\end{cor}

\begin{proof}
(1) 
Letting $n=0$, we have that the bijection in Theorem \ref{31}(1) is nothing but the first one-to-one correspondence.
Indeed, Remark \ref{rem2}(2), Propositions \ref{3}, and \cite[Lemma 3.5]{Kra08} show that $0$-wide $(\E, \oplus,\Pi)$-closed subcategories of $\Mod R$ are the same things as bilocalizing subcategories of $\Mod R$.
By Proposition \ref{3}, the $0$-coherent generalization closed subsets of $\Spec R$ are the clopen subsets of $\Spec R$.
Using Proposition \ref{27}, we see that the bijection in Theorem \ref{31}(2) gives a one-to-one correspondence between the set of clopen subsets of $\Spec R$ and the set of smashing bilocalizing subcategories of $\D(\Mod R)$.
Thus, it suffices to check that every smashing bilocalizing subcategory $\X$ of $\D(\Mod R)$ is bismashing, i.e., the inclusion functor $\X \hookrightarrow \D(\Mod R)$ has a left adjoint which preserves direct products.
There is a clopen subset $\Phi$ of $\Spec R$ with $\X=\supp^{-1}_{\D(\Mod R)}(\Phi)$.
Set $\Y= \supp^{-1}_{\D(\Mod R)}(\Phi^\complement)$.
By \cite[Proposition 4.9.1]{Kra10} and \cite[Lemma 4.5]{Nee11}, the inclusion functor $\X = \Y^\perp \hookrightarrow \D(\Mod R)$ has a left adjoint $\D(\Mod R) \to \D(\Mod R)/\Y \cong \X$.
Using Proposition 6.1, we see that $\Y= \supp^{-1}_{\D(\Mod R)}(\Phi^\complement)$ is closed under direct products, and hence this left adjoint preserves direct products.
%Unify Remarks \ref{rem2}, \ref{4}, \ref{28}, Theorems \ref{7}, \ref{8-2}, \ref{31} and Corollaries \ref{infcoh}, \ref{36'}, \ref{35}, \ref{25}, \ref{34}.

(2) 
The assertion follows by Remarks \ref{rem2}(3), \ref{4}(3), \ref{28}(2), \ref{32}, \cite[Lemma 3.5]{Kra08} and letting $n=1$ in Theorem \ref{31}.
Indeed, Remark \ref{rem2}(3) and \cite[Lemma 3.5]{Kra08} show that the $1$-wide $(\E, \oplus, \Pi)$-closed subcategories are the wide $(\oplus, \Pi)$-closed subcategories which are by Remark \ref{32} the bireflective Giraud subcategories of $\Mod R$.
By Remark \ref{28}(2), the $1$-uniform bilocalizing subcategories are the cohomology-closed bilocalizing subcategories of $\D(\Mod R)$ and by Remark \ref{4}(3) the $1$-coherent generalization-closed subsets are the generalization-closed coherent subsets of $\Spec R$.
\if0
\old{
Apply Remark \ref{rem2}(2), Propositions \ref{3}, \ref{27}, Corollary \ref{33}(3), \cite[Lemma 3.5]{Kra08}, and let $n=0$ in Theorem \ref{31}.
Then we see that it suffices to check that for every smashing bilocalizing subcategory $\X$ of $\D(\Mod R)$, the inclusion functor $\X \hookrightarrow \D(\Mod R)$ has a left adjoint preserving direct products.
There is a clopen subset $\Phi$ of $\Spec R$ with $\X=\supp^{-1}_{\D(\Mod R)}(\Phi)$.
Set $\Y= \supp^{-1}_{\D(\Mod R)}(\Phi^\complement)$.
By \cite[Proposition 4.9.1]{Kra10} and \cite[Lemma 4.5]{Nee11}, the inclusion functor $\X = \Y^\perp \hookrightarrow \D(\Mod R)$ has a left adjoint $\D(\Mod R) \to \D(\Mod R)/\Y \cong \X$ which preserves direct products, since $\supp \Y = \Phi^\complement$ is generalization-closed.
Unify Remarks \ref{rem2}, \ref{4}, \ref{28}, Theorems \ref{7}, \ref{8-2}, \ref{31} and Corollaries \ref{infcoh}, \ref{36'}, \ref{35}, \ref{25}, \ref{34}.

(2) The assertion follows by Remarks \ref{rem2}(3), \ref{4}(3), \ref{28}(2), \ref{32}, \cite[Lemma 3.5]{Kra08} and letting $n=1$ in Theorem \ref{31}.
}
\fi
\end{proof}

Next, we consider classifying certain thick subcategories of $\D(\Mod R)_{\fid}$, $\Db(\Mod R)$ and $\Dp(\Mod R)$.
Here, $\D(\Mod R)_{\fid}$ stands for the subcategory of $\D(\Mod R)$ consisting of complexes of finite injective dimension.
Recall that $\oplus$-closedness (resp. $\Pi$-closedness) means closed under {\it existing} coproducts (resp. products) for a subcategory of $\D(\Mod R)_{\fid}$, $\Db(\Mod R)$ or $\Dp(\Mod R)$.

\begin{thm}\label{40}
Let $n\in\NN\cup\{\infty\}$.
\begin{enumerate}[\rm(1)]
\item
The assignments $\X\mapsto\supp\X$ and $\supp^{-1}_{\Dp(\Mod R)}(\Phi)\mapsfrom\Phi$ give a bijective correspondence between the $n$-uniform $(\E, \oplus, \Pi)$-closed thick subcategories of $\Dp(\Mod R)$ and the $n$-coherent generalization-closed subsets of $\Spec R$.
\item
The assignments $\X\mapsto\supp\X$ and $\supp^{-1}_{\D(\Mod R)_{\fid}}(\Phi) \mapsfrom\Phi$ give a bijective correspondence between the $n$-uniform $(\E, \oplus)$-closed thick subcategories of $\D(\Mod R)_{\fid}$ and the $n$-coherent subsets of $\Spec R$.
\end{enumerate}
\end{thm}

\begin{proof}
We begin with establishing two claims.

\begin{claim}\label{c1}
It suffices to show the assertion for $n=\infty$.
\end{claim}

\begin{proof}[Proof of Claim]
Assume that this claim is shown to hold.
We prove that (1) and (2) in the theorem hold for $n\in\NN$.

(1) Let $\X$ be an $n$-uniform $(\E, \oplus, \Pi)$-closed thick subcategory of $\Dp(\Mod R)$ and $\Phi$ an $n$-coherent generalization-closed subset of $\Spec R$.
Then $\X$ is $\infty$-uniform and $\Phi$ is $\infty$-coherent.
By assumption, $\supp\X$ is generalization-closed, and $\supp^{-1}_{\Dp(\Mod R)}(\Phi)$ is thick and $(\E, \oplus, \Pi)$-closed.
%\old{There is a generalization-closed subset $\Psi$ of $\Spec R$ such that $\X = \supp^{-1}_{\Dp(\Mod R)}(\Psi)$.
%Each $X\in\X$ is isomorphic to the complex $\E(X)$, and $\E^i(X)$ is in $\X$ for all $i$.
%Hence the same argument as in the proof of Proposition \ref{30} works, and t}
The proof of Theorem \ref{8-2} shows that $\supp\X$ is $n$-coherent and $\supp^{-1}_{\Dp(\Mod R)}(\Phi)$ is $n$-uniform.
Thus the assignments are well-defined and bijective.

(2) This is shown in a similar way to the proof of (1).
\renewcommand{\qedsymbol}{$\square$}
\end{proof}

\begin{claim}\label{c2}
Let $\X$ be an $(\E, \Pi)$-closed thick subcategory of $\Dp(\Mod R)$ or an $\E$-closed thick subcategory of $\D(\Mod R)_{\fid}$.
Then $\X$ is $\ee$-closed. 	
\end{claim}

\begin{proof}[Proof of Claim]
The second case is obvious as every complex $X\in\D(\Mod R)_{\fid}$ is isomorphic to the bounded complex $\E(X)$.
Let $\X$ be an $(\E, \Pi)$-closed thick subcategory of $\Dp(\Mod R)$, and take a complex $X \in \Dp(\Mod R)$ with $\E^i(X) \in \X$ for all $i \in \Z$.
Let $\E(X)^{\le i}:= (\cdots \to \E^{i-1}(X) \to \E^i(X) \to 0)$ be a (hard) truncation of $\E(X)$ for each $i$.
As $\E(X)^{\le i}$ is a bounded complex of objects in $\X$, it is in $\X$.
The canonical morphisms $X \cong \E(X) \to \E(X)^{\le i}$ for $i \in \Z$ induce an isomorphism $X \cong \holim_{i \in \Z} \E(X)^{\le i}$ in $\Dp(\Mod R)$; see \cite[Remark 2.3]{BN}.
Thus there is an exact triangle $X \to \Pi_{i \in \Z} \E(X)^{\le i} \to \Pi_{i \in \Z} \E(X)^{\le i} \to X[1]$ in $\Dp(\Mod R)$.
As $\E(X)^{\le i} \in \X$ and $\X$ is a $\Pi$-closed thick subcategory, it is seen that $X\in\X$.
\renewcommand{\qedsymbol}{$\square$}
\end{proof}

Using Remark \ref{rem}(1), we get $\supp(\supp^{-1}_{\Dp(\Mod R)}(\Phi)) = \Phi$ and $\supp(\supp^{-1}_{\D(\Mod R)_{\fid}}(\Phi)) = \Phi$ for any subset $\Phi$ of $\Spec R$.
Let $\X$ be an $(\E, \oplus, \Pi)$-closed thick subcategory of $\Dp(\Mod R)$ (resp. an $(\E, \oplus)$-closed thick subcategory of $\D(\Mod R)_{\fid}$).
Let $X$ be a complex in $\supp^{-1}_{\Dp(\Mod R)}(\supp \X)$ (resp. $\supp^{-1}_{\D(\Mod R)_{\fid}}(\supp \X)$).
Then $\E^i(X) \in \X$ for all $i \in \Z$ by Remark \ref{rem} and the proof of Theorem \ref{7}. 
Claim \ref{c2} shows $X \in \X$.
Hence $\supp^{-1}_{\Dp(\Mod R)}(\supp \X) = \X$ (resp. $\supp^{-1}_{\D(\Mod R)_{\fid}}(\supp \X) = \X$), and the proof of the theorem for $n=\infty$ is completed.
Combining this with Claim \ref{c1}, we are done.
\end{proof}

\begin{rem}
By tracing the argument in Section 5, the above classifications show that if $0\ne n \in \NN \cup \{\infty\}$, then $n$-uniformity and $n$-consistency are equivalent for $(\E, \oplus, \Pi)$-closed thick subcategories of $\Dp(\Mod R)$ and for $(\E, \oplus)$-closed thick subcategories of $\D(\Mod R)_{\fid}$; see the proofs of Proposition \ref{8-4} and Theorem \ref{24}.
\end{rem}

As an application of Theorem \ref{40} together with Corollary \ref{10}, we obtain a higher-dimensional analogue of Br\"{u}ning's classification theorem \cite[Theorem 5.1]{Br} for the module category.

\begin{cor}\label{50}
Let $n \in \NN$.
Assume that $\gldim R \le n$.
Then any $(\E,\oplus)$-closed thick subcategory of $\Db(\Mod R)$ is $n$-uniform.
Hence, there are one-to-one correspondences among the $(\E, \oplus)$-closed thick subcategories of $\Db(\Mod R)$, the $(\E, \oplus)$-closed $n$-wide subcategories of $\Mod R$, and the subsets of $\Spec R$. 
\end{cor}

\begin{proof}
Since $\gldim R \le n$, the ring $R$ has Krull dimension at most $n$ and $\D(\Mod R)_{\fid} \cong \Db(\Mod R)$.
Thus, the first statement follows from Corollary \ref{10} and Theorem \ref{40}(2).
The second statement is a consequence of the combination of the first statement and Theorems \ref{40}(2), \ref{7}.
\end{proof}

Taking $n=1$ in the above corollary, we obtain the following.

\begin{cor}
Assume that $R$ is hereditary, i.e., $\gldim R \le 1$.
Then, there are one-to-one correspondences among the $\oplus$-closed thick subcategories of $\Db(\Mod R)$, the $\oplus$-closed wide subcategories of $\Mod R$, and the subsets of $\Spec R$. 
\end{cor}

\begin{proof}
In view of Corollaries \ref{50} and \ref{36'}, it is enough to show that every $\oplus$-closed ($1$-uniform) thick subcategory $\X$ of $\Db(\Mod R)$ is $\E$-closed.
The proof is almost the same as \cite[Lemma 3.5]{Kra08}.
We need to prove that for all $X \in \X$ and $\p \in \supp X$ it holds that $\e(R/\p) \in \X$.
By the definition of a small support, $\Tor_t^R(X, \kappa(\p)) \neq 0$ for some $t\in\Z$.
Since $R$ is hereditary, $\kappa(\p)$ has a projective $R$-resolution of length at most one.
As $\X$ is thick, $X \ltensor_R \kappa(\p)$ belongs to $\X$.
Since $X$ is $1$-uniform, it is closed under cohomologies.
Hence $\Tor_t^R(X, \kappa(\p))$ is in $\X$, and so is its direct summand $\kappa(\p)$.
Note that $\e(R/\p)$ is a directed union of direct sums of copies of $\kappa(\p)$.
It follows that $\e(R/\p) \in \X$.
\end{proof}

Let $X$ be an $R$-complex.
For each integer $n$, we define {\em (soft) truncations}
$$
\tau^{\le n}X=(\cdots\to X^{n-2}\to X^{n-1}\to\ZZ^nX\to0),\qquad
\tau^{\ge n}X=(0\to \BB^nX\to X^n\to X^{n+1}\to\cdots)
$$
of $X$, where $\ZZ^nX,\BB^nX$ are placed in degree $n,n-1$ respectively.
Note that there is a natural exact sequence $0\to\tau^{\le n}X\to X\to\tau^{\ge n+1}X\to0$ of complexes.
The following result is worth comparing with Corollary \ref{50}.

\begin{prop}
Let $n\in\NN$.
When $\gldim R \le n$, every thick subcategory of $\D(\Mod R)$ is $n$-consistent.
\end{prop}

\begin{proof}
First of all, we claim that for an integer $k$ and a complex $P$ of projective $R$-modules with $\H^jP=0$ for all $k+1\le j\le k+n-1$ there is a direct sum decomposition $P\cong\tau^{\le k}P\oplus\tau^{\ge k+1}P$ of $R$-complexes.
In fact, since $\gldim R\le n$, the exact sequence $0\to\BB^{k+1}P\to P^{k+1}\to\cdots\to P^{k+n}$ shows that the $R$-module $\BB^{k+1}P$ is projective.
Hence the exact sequence $0\to\ZZ^kP\to P^k\to\BB^{k+1}P\to0$ splits, and so does the exact sequence $0\to\tau^{\le k}P\to P\to\tau^{\ge k+1}P\to0$.

Let $\X$ be a thick subcategory of $\D(\Mod R)$.
Let $X \in \X$ be such that $\H^jX = 0$ for all $j$ with $i \neq j \in (i-n, i+n)$.
Replacing $X$ with a semiprojective resolution, we may assume that $X^j$ is projective for each $j\in\Z$.
%\old{The assumption and the exact sequence $0 \to \BB^{i+1}X \to X^{i+1} \to X^{i+2} \to \cdots \to X^{i+n}$ imply that $\BB^{i+1}X$ is projective.
%Therefore, the short exact sequence $0 \to \ZZ^iX \to X^i \to \BB^{i+1}X \to 0$ splits and hence there is a direct sum decomposition}
As $\H^jX$ (resp. $\H^j(\tau^{\le i}X)$) vanishes for all $i+1\le j\le i+n-1$ (resp. $(i-n)+1\le j\le (i-n)+n-1$), the claim shows that there are direct sum decompositions $X \cong \tau^{\le i}X \oplus \tau^{\ge i+1} X$ and $\tau^{\le i}X \cong \tau^{\le i-n} \tau^{\le i}X \oplus \tau^{\ge i-n+1} \tau^{\le i}X \cong \tau^{\le i-n} X \oplus \H^iX[-i]$.
Here, to apply the claim we replace $\tau^{\le i}X$ with a projective resolution $(\cdots\to Q^{i-2}\to Q^{i-1}\to Q^i\to0)$, so that each component is a projective $R$-module.
Thus we get $X \cong \tau^{\le i-n} X \oplus \H^iX[-i] \oplus \tau^{\ge i+1} X$.
As $\X$ is closed under direct summands and shifts, we obtain $\H^iX \in \X$.
\end{proof}

%%%%%%%%%%%%%%%%%%%%%%%%%%%%%%%%%%%%%%%%%%%%%%%%%%%%%%%
\section{A remark on cofinite modules}

For an ideal $J$ over a commutative noetherian ring $R$, Hartshorne \cite{H} introduced the notion of a $J$-cofinite module and posed a question of whether the category of $J$-cofinite modules forms an abelian category.
A counterexample to this question is given by himself.
However, the question is still considered by several authors in the direction of finding conditions that lead the question to be affirmative, e.g., \cite{BNS, DFT, Me1, Me2}.
In this section, we consider a weakened version of cofinite modules and $n$-coherence of the category of such modules.
Although our result does not recover known results, it may provide another approach to the question.

By \cite[Corollary 3.5]{DFT}, cofiniteness of modules are related to artinianness of local cohomology modules.
Our strategy is to relax artinianness as follows.

\begin{dfn}
We say that an $R$-module $M$ is {\em ind-artinian} if all finitely generated submodules of $M$ have finite length.
Note that any artinian $R$-module is ind-artinian and an $R$-module $M$ is ind-artinian if and only if it is an inductive limit of artinian modules.
\end{dfn}

The following proposition shows that the family of ind-artinian modules forms a localizing subcategory.
More strongly, it is the smallest localizing subcategory of $\Mod R$ containing all artinian modules.

\begin{prop}\label{923-5}
The subcategory of $\Mod R$ consisting of ind-artinian modules coincides with $\Ass_{\Mod R}^{-1}(\Max R)$.
This also coincides with $\supp_{\Mod R}^{-1}(\Max R)$, and is localizing.
\end{prop}

\begin{proof}
Let $M$ be an $R$-module.
If $\p\in\Ass M$ is a nonmaximal prime ideal, then $M$ has a submodule isomorphic to $R/\p$, which does not have finite length.
This shows that every ind-artinian $R$-module belongs to $\Ass_{\Mod R}^{-1}(\Max R)$.
Conversely, assume $\Ass M\subseteq\Max R$.
If $N$ is a finitely generated submodule of $M$, then $\Ass N\subseteq\Ass M\subseteq\Max R$.
Since the minimal primes of $M$ are associated primes, we see that $\Supp N\subseteq\Max R$.
As $N$ is finitely generated, it has finite length.
Thus the proof of the first assertion of the proposition is completed.
The second assertion follows from \cite[Theorem 2.6]{cd}.
\end{proof}

The following lemma generalizes \cite[Proposition 4.5]{cd} in a relative setting.

\begin{lem}\label{923-3}
Let $\Phi$ be a specialization-closed subset of $\Spec R$.
Let $M$ be an $R$-module.
\begin{enumerate}[\rm(1)]
\item
There is an equality $\Ass_R\Gamma_\Phi(M)=\Ass_RM\cap\Phi$.
\item
Let $n\ge0$ be an integer.
\begin{enumerate}[\rm(a)]
\item
Let $\Psi$ be a subset of $\Spec R$.
Then $\H_\Phi^i(M)\in\supp_{\Mod R}^{-1}(\Psi)$ for all $i\le n$ if and only if $\Gamma_\Phi(\E^i(M))\in\supp_{\Mod R}^{-1}(\Psi)$ for all $i\le n$.
\item
Let $\X$ be an $(\ee,\oplus)$-closed subcategory of $\Mod R$.
Then $\H_\Phi^i(M)\in\X$ for all $i\le n$ if and only if $\Gamma_\Phi(\E^i(M))\in\X$ for all $i\le n$.
\end{enumerate}
\end{enumerate}
\end{lem}

\begin{proof}
(1) We use \cite[Propositions 3.2(1a) and 2.3(6)]{cd} for both inclusions.
Let $\p\in\Ass M\cap\Phi$.
Then there is a monomorphism $R/\p\hookrightarrow M$, which induces a monomorphism $\Gamma_\Phi(R/\p)\hookrightarrow\Gamma_\Phi(M)$.
As $\V(\p)\subseteq\Phi$, we have $\Gamma_\Phi(R/\p)=R/\p$.

(2a) The equality $\H_\Phi^i(M)=\H^i(\E(M))$ and Proposition \ref{923-5} deduce the ``if'' part.
We prove the ``only if'' part by induction on $n$.
We have
$$
\supp\Gamma_\Phi(\E^0(M))
=\Ass\Gamma_\Phi(\E^0(M))
=\Ass\E^0(M)\cap\Phi
=\Ass M\cap\Phi
=\Ass\Gamma_\Phi(M)
\subseteq\supp\Gamma_\Phi(M)
\subseteq\Psi,
$$
where the first equality holds since $\Gamma_\Phi(\E^0(M))$ is injective, and the second and fourth ones follow from Lemma \ref{923-3}(1).
Now let $n>0$, and take an exact sequence $0\to M\to\e(M)\to\mho M\to0$.
Then $\H_\Phi^i(\mho M)=\H_\Phi^{i+1}(M)\in\supp^{-1}_{\Mod R}(\Psi)$ for all $1\le i\le n-1$.
An exact sequence $0\to\Gamma_\Phi(M)\to\Gamma_\Phi(\E^0(M))\to\Gamma_\Phi(\mho M)\to\H_\Phi^1(M)\to0$ is induced.
The modules $\Gamma_\Phi(M),\H_\Phi^1(M)$ belong to $\supp_{\Mod R}^{-1}(\Psi)$ by assumption, while so does $\Gamma_\Phi(\E^0(M))$ by the induction basis.
It is observed that $\Gamma_\Phi(\mho M)\in\supp_{\Mod R}^{-1}(\Psi)$.
The induction hypothesis implies that $\Gamma_\Phi(\E^i(M))=\Gamma_\Phi(\E^{i-1}(\mho M))\in\supp^{-1}_{\Mod R}(\Psi)$ for all $1\le i\le n$.
Thus the proof of the assertion is completed.

(2b) The assertion is a direct consequence of (2a) and Corollary \ref{36'}.
\end{proof}

\begin{prop}\label{923-4}
Let $\Phi,\Psi$ be subsets of $\Spec R$.
Suppose that $\Phi$ is specialization-closed.
Then $\supp_{\Mod R}^{-1}(\Phi^\complement\cup\Psi)$ coincides with the subcategory of $\Mod R$ consisting of modules $M$ with $\H_\Phi^i(M)\in\supp^{-1}_{\Mod R}(\Psi)$ for all $i\ge0$.
\end{prop}

\begin{proof}
Let $M$ be an $R$-module.
Using Lemma \ref{923-3}(2a), we have that $\H_\Phi^i(M)\in\supp^{-1}_{\Mod R}(\Psi)$ for all $i\ge0$ if and only if $\Psi$ contains $\supp\Gamma_\Phi(\E^i(M))=\Ass\Gamma_\Phi(\E^i(M))=\Ass\E^i(M)\cap\Phi$ for all $i\ge0$, where the first equality holds since $\Gamma_\Phi(\E^i(M))$ is injective, and the second one follows from Lemma \ref{923-3}(1).
The latter condition is equivalent to saying that $\Phi^\complement\cup\Psi$ contains $\Ass\E^i(M)$ for all $i\ge0$, which is equivalent to the inclusion $\Phi^\complement\cup\Psi\supseteq\supp M$.
\end{proof}

\begin{rem}
Putting $\X=0$ in Lemma \ref{923-3}(2) and $\Psi=\emptyset$ in Proposition \ref{923-4}, we recover \cite[Proposition 4.5]{cd} and \cite[Remark 3.8(1)]{cd}, respectively.
\end{rem}

For a specialization-closed subset $\Phi$ of $\Spec R$, we denote by $\hh_\Phi$ the subcategory of $\Mod R$ consisting of modules $M$ such that $\H_\Phi^i(M)$ is ind-artinian for all $i\ge0$.

\begin{cor}\label{923-2}
Let $\Phi$ be a specialization-closed subset of $\Spec R$.
\begin{enumerate}[\rm(1)]
\item
There is an equality $\hh_\Phi=\supp^{-1}_{\Mod R}(\Phi^\complement\cup\Max R)$.
\item
Let $R$ have finite positive Krull dimension $d$.
Then $\hh_\Phi$ is $(d-1)$-wide.
In particular, $\hh_\Phi$ is abelian if $d\le2$.
\end{enumerate}
\end{cor}

\begin{proof}
(1) Apply Proposition \ref{923-4} for $\Psi=\Max R$ and use Proposition \ref{923-5}.

(2) We have $\height\p<d$ for all $\p\in\Phi\setminus\Max R$.
Proposition \ref{923-1}(1) implies that $(\Phi\setminus\Max R)^\complement=\Phi^\complement\cup\Max R$ is $(d-1)$-coherent.
It follows from Corollary \ref{923-2}(1) and Theorem \ref{7} that $\hh_\Phi$ is $(d-1)$-wide.
Thus the proof of the first assertion is completed.
The second assertion follows from the first one and \cite[Remark 4.2(5)]{cd}.
\end{proof}

\begin{dfn}
Let $\Phi$ be a specialization-closed subset of $\Spec R$.
An $R$-module $M$ is said to be {\em $\Phi$-cofinite} if $\Supp_RM\subseteq\Phi$ and the $R$-module $\Ext_R^i(R/I,M)$ is finitely generated for all integers $i$ and ideals $I$ of $R$ with $\V(I)\subseteq\Phi$.
We denote by $\cof\Phi$ the subcategory of $\Mod R$ consisting of $\Phi$-cofinite $R$-modules.
\end{dfn}

If $R$ is semilocal, then $\cof\Phi=\{M\in\Supp_{\Mod R}^{-1}(\Phi)\mid\text{$\H_\Phi^i(M^\vee)$ is artinian for all $i\ge0$}\}$ by \cite[Corollary 3.5]{DFT}, where $(-)^\vee=\Hom_R(-,\e_R(R/\rad R))$.
As an analogue, (for a not-necessarily-semilocal ring $R$) we define
$$
\cof'\Phi=\{M\in\Supp_{\Mod R}^{-1}(\Phi)\mid\text{$\H_\Phi^i(M^\vee)$ is ind-artinian for all $i\ge0$}\}.
$$
It is proved in \cite[Theorem 3.13(iii)]{DFT} that if $\dim R\le2$, then $\cof\Phi$ is abelian for a specialization-closed subset $\Phi$ of $\Spec R$.
We can prove an analogue of this result for $\cof'\Phi$.

\begin{cor}
One has $\cof'\Phi=\{M\in\supp^{-1}_{\Mod R}(\Phi)\mid M^\vee\in\supp_{\Mod R}^{-1}(\Phi^\complement\cup\Max R)\}$ for a specialization-closed subset $\Phi$ of $\Spec R$.
If $\dim R\le2$, then $\cof'\Phi$ is abelian.
\end{cor}

\begin{proof}
We use \cite[Theorem 2.6]{cd} and Corollary \ref{923-2}.
Both $\supp^{-1}_{\Mod R}(\Phi)$ and $\supp_{\Mod R}^{-1}(\Phi^\complement\cup\Max R)$ are abelian, and it is routine to deduce the last assertion from this.
\end{proof}

%%%%%%%%%%%%%%%%%%%%%%%%%%%%%%%%%%%%%%%%%%%%%%%%%%

\begin{ac}
The authors thank Tsutomu Nakamura for reminding them of Lemma \ref{23} and the correctness of \cite[Theorem 5.2]{Kra08}; without this, the authors would not be able to get the proof of Proposition \ref{8-4}.
The authors also thank Lidia Angeleri H\"ugel for giving them useful observation on bireflective Giraud subcategories.
Part of this work was done during Takahashi's visit to the University of Kansas from March 2018 to September 2019, and he thanks them for their hospitality.
Also, the authors very much thank the anonymous referees for their careful reading, valuable comments and helpful suggestions.
\end{ac}
%%%%%%%%%%%%%%%%%%%%%%%%%%%%%%%%%%%%%%%%%%%%%%%%%%

%%%%%%%%%%%%%%%%%%%%%%%%%%%%%%%%%%%%%%%%%%%%%%%%%%%%
%%%%%%%%%%%%%%%%%%%%%%%%%%%%%%%%%%%%%%%%%%%%%%%%%%%%%%
\end{document}